\documentstyle[10pt,amscd,xypic,amssymb]{amsart}

\xyoption{all}
\CompileMatrices

\newcommand{\nc}{\newcommand}

\makeatletter
\@addtoreset{equation}{section}
\makeatother

\newenvironment{proof}{{\noindent \textbf{Proof}\,\,}}{\hspace*{\fill}$\Box$\medskip}

\newtheorem{theorem}[equation]{Theorem}
\newtheorem{proposition}[equation]{Proposition}
\newtheorem{lemma}[equation]{Lemma}
\newtheorem{corollary}[equation]{Corollary}

\theoremstyle{definition}

\newtheorem{definition}[equation]{Definition}

\theoremstyle{remark}

\newtheorem{remark}[equation]{Remark}

\nc{\BA}{{\mathbb{A}}}
\nc{\BC}{{\mathbb{C}}}
\nc{\BM}{{\mathbb{M}}}
\nc{\BN}{{\mathbb{N}}}
\nc{\Bk}{{\Bbbk}}
\nc{\BF}{{\mathbb{F}}}
\nc{\BP}{{\mathbb{P}}}
\nc{\BR}{{\mathbb{R}}}
\nc{\BZ}{{\mathbb{Z}}}

\nc{\Mod}{{{\mathcal M}od}}

\nc{\CA}{{\mathcal{A}}}
\nc{\CB}{{\mathcal{B}}}
\nc{\D}{{\mathcal{D}}}
\nc{\CE}{{\mathcal{E}}}
\nc{\CF}{{\mathcal{F}}}
\nc{\CG}{{\mathcal{G}}}
\nc{\CL}{{\mathcal{L}}}
\nc{\CM}{{\mathcal{M}}}
\nc{\CN}{{\mathcal{N}}}
\nc{\CO}{{\mathcal{O}}}
\nc{\CP}{{\mathcal{P}}}
\nc{\CQ}{{\mathcal{Q}}}
\nc{\CS}{{\mathcal{S}}}
\nc{\CT}{{\mathcal{T}}}
\nc{\CU}{{\mathcal{U}}}
\nc{\CV}{{\mathcal{V}}}
\nc{\CW}{{\mathcal{W}}}
\nc{\IC}{{\mathcal{IC}}}

\nc{\cM}{{\check{\mathcal M}}{}}
\nc{\csM}{{\check{\mathcal A}}{}}
\nc{\oM}{{\overset{\circ}{\mathcal M}}{}}
\nc{\obM}{{\overset{\circ}{\mathbf M}}{}}
\nc{\oCA}{{\overset{\circ}{\mathcal A}}{}}
\nc{\obA}{{\overset{\circ}{\mathbf A}}{}}
\nc{\ooM}{{\overset{\circ}{M}}{}}
\nc{\osM}{{\overset{\circ}{\mathsf M}}{}}
\nc{\vM}{{\overset{\bullet}{\mathcal M}}{}}
\nc{\nM}{{\underset{\bullet}{\mathcal M}}{}}
\nc{\oD}{{\overset{\circ}{\mathcal D}}{}}
\nc{\obD}{{\overset{\circ}{\mathbf D}}{}}
\nc{\oA}{{\overset{\circ}{\mathbb A}}{}}
\nc{\op}{{\overset{\bullet}{\mathbf p}}{}}
\nc{\cp}{{\overset{\circ}{\mathbf p}}{}}
\nc{\oU}{{\overset{\bullet}{\mathcal U}}{}}

\nc{\fa}{{\mathfrak{a}}}
\nc{\fb}{{\mathfrak{b}}}
\nc{\fg}{{\mathfrak{g}}}
\nc{\fh}{{\mathfrak{h}}}
\nc{\fj}{{\mathfrak{j}}}
\nc{\fn}{{\mathfrak{n}}}
\nc{\fu}{{\mathfrak{u}}}
\nc{\fp}{{\mathfrak{p}}}
\nc{\fr}{{\mathfrak{r}}}
\nc{\ft}{{\mathfrak{t}}}
\nc{\fsl}{{\mathfrak{sl}}}
\nc{\hsl}{{\widehat{\mathfrak{sl}}}}
\nc{\hgl}{{\widehat{\mathfrak{gl}}}}
\nc{\hg}{{\widehat{\mathfrak{g}}}}
\nc{\chg}{{\widehat{\mathfrak{g}}}{}^\vee}
\nc{\hn}{{\widehat{\mathfrak{n}}}}
\nc{\chn}{{\widehat{\mathfrak{n}}}{}^\vee}

\nc{\fA}{{\mathfrak{A}}}
\nc{\fB}{{\mathfrak{B}}}
\nc{\fC}{{\mathfrak{C}}}
\nc{\fD}{{\mathfrak{D}}}
\nc{\fE}{{\mathfrak{E}}}
\nc{\fF}{{\mathfrak{F}}}
\nc{\fG}{{\mathfrak{G}}}
\nc{\fK}{{\mathfrak{K}}}
\nc{\fL}{{\mathfrak{L}}}
\nc{\fM}{{\mathfrak{M}}}
\nc{\fN}{{\mathfrak{N}}}
\nc{\fP}{{\mathfrak{P}}}
\nc{\fU}{{\mathfrak{U}}}

\nc{\bb}{{\mathbf{b}}}
\nc{\bc}{{\mathbf{c}}}
\nc{\be}{{\mathbf{e}}}
\nc{\bj}{{\mathbf{j}}}
\nc{\bn}{{\mathbf{n}}}
\nc{\bp}{{\mathbf{p}}}
\nc{\bq}{{\mathbf{q}}}
\nc{\bu}{{\mathbf{u}}}
\nc{\bv}{{\mathbf{v}}}
\nc{\bx}{{\mathbf{x}}}
\nc{\by}{{\mathbf{y}}}
\nc{\bw}{{\mathbf{w}}}
\nc{\bA}{{\mathbf{A}}}
\nc{\bB}{{\mathbf{B}}}
\nc{\bC}{{\mathbf{C}}}
\nc{\bD}{{\mathbf{D}}}
\nc{\bH}{{\mathbf{H}}}
\nc{\bM}{{\mathbf{M}}}
\nc{\bV}{{\mathbf{V}}}
\nc{\bW}{{\mathbf{W}}}
\nc{\bX}{{\mathbf{X}}}

\nc{\sA}{{\mathsf{A}}}
\nc{\sB}{{\mathsf{B}}}
\nc{\sD}{{\mathsf{D}}}
\nc{\sF}{{\mathsf{F}}}
\nc{\sK}{{\mathsf{K}}}
\nc{\sM}{{\mathsf{M}}}
\nc{\sO}{{\mathsf{O}}}
\nc{\sQ}{{\mathsf{Q}}}
\nc{\sP}{{\mathsf{P}}}
\nc{\sfp}{{\mathsf{p}}}
\nc{\sr}{{\mathsf{r}}}

\nc{\BK}{{\bar{K}}}

\nc{\tA}{{\widetilde{\mathbf{A}}}}
\nc{\TG}{{\tilde{G}}}
\nc{\TM}{{\widetilde{\mathbb{M}}}{}}
\nc{\tO}{{\widetilde{\mathsf{O}}}{}}
\nc{\TZ}{{\tilde{Z}}}
\nc{\tx}{{\tilde{x}}}
\nc{\tbv}{{\tilde{\bv}}}
\nc{\tfP}{{\widetilde{\mathfrak{P}}}{}}
\nc{\tz}{{\tilde{\zeta}}}
\nc{\tmu}{{\tilde{\mu}}}

\nc{\urho}{\underline{\rho}}
\nc{\uB}{\underline{B}}
\nc{\uC}{{\underline{\mathbb{C}}}}
\nc{\ui}{\underline{i}}
\nc{\uj}{\underline{j}}
\nc{\ofP}{{\overline{\mathfrak{P}}}}

\nc{\eps}{\varepsilon}
\nc{\hrho}{{\hat{\rho}}}

\nc{\one}{{\mathbf{1}}}
\nc{\two}{{\mathbf{t}}}

\nc{\Rep}{{\mathop{\operatorname{\rm Rep}}}}
\nc{\Sym}{{\mathop{\operatorname{\rm Sym}}}}
\nc{\Tot}{{\mathop{\operatorname{\rm Tot}}}}
\nc{\Spec}{{\mathop{\operatorname{\rm Spec}}}}
\nc{\Ker}{{\mathop{\operatorname{\rm Ker}}}}
\nc{\Hilb}{{\mathop{\operatorname{\rm Hilb}}}}
\nc{\End}{{\mathop{\operatorname{\rm End}}}}
\nc{\Ext}{{\mathop{\operatorname{\rm Ext}}}}
\nc{\Hom}{{\mathop{\operatorname{\rm Hom}}}}
\nc{\CHom}{{\mathop{\operatorname{{\mathcal{H}}\it om}}}}
\nc{\GL}{{\mathop{\operatorname{\rm GL}}}}
\nc{\gr}{{\mathop{\operatorname{\rm gr}}}}
\nc{\Id}{{\mathop{\operatorname{\rm Id}}}}
\nc{\rk}{{\mathop{\operatorname{\rm r}}}}
\nc{\de}{{\mathop{\operatorname{\rm def}}}}
\nc{\length}{{\mathop{\operatorname{\rm length}}}}
\nc{\supp}{{\mathop{\operatorname{\rm supp}}}}

\nc{\Bun}{{\mathsf{Bun}}}
\nc{\Cliff}{{\mathsf{Cliff}}}
\nc{\Gr}{{\mathsf{Gr}}}
\nc{\Fl}{{\mathsf{Fl}}}
\nc{\Fib}{{\mathsf{Fib}}}
\nc{\Coh}{{\mathsf{Coh}}}
\nc{\FCoh}{{\mathsf{FCoh}}}

\nc{\reg}{{\text{\rm reg}}}

\nc{\cplus}{{\mathbf{C}_+}}
\nc{\cminus}{{\mathbf{C}_-}}
\nc{\cthree}{{\mathbf{C}_*}}
\nc{\Qbar}{{\bar{Q}}}

\nc{\bh}{{\bar{h}}}
\nc{\bOmega}{{\overline{\Omega}}}

\nc{\seq}[1]{\stackrel{#1}{\sim}}

\def\dsp{\displaystyle}

\begin{document}

\title[Laumon Spaces and the Calogero-Sutherland Integrable System]{\Large{\textbf{Laumon Spaces and the Calogero-Sutherland Integrable System}}}

\author[Andrei Negut]{Andrei Negut}
\address{Princeton University, Department of Mathematics, Princeton, NJ 08544, USA}
\email{andrei.negut@@gmail.com}

\maketitle

\section{Introduction}

This paper is concerned with the Laumon quasiflag spaces $\CM_\gamma$, which parametrize flags of the form
\begin{equation}
0\subset \CF_1\subset \CF_2\subset...\subset \CF_{n-1}\subset \CO_{\BP^1}^n
\end{equation}
In the above, each $\CF_i$ is a torsion-free sheaf of rank $i$ on $\BP^1$ whose fiber at $\infty\in \BP^1$ is fixed beforehand. The degrees of the sheaves $\CF_i$ are predetermined by the index $\gamma$, which lies in a monoid $Q^-$ in the root lattice of $\fsl_n$ (see Section~\ref{Quasiflag Spaces} for the exact definitions). The maximal torus $T\subset SL_n$ acts on each $\CM_{\gamma}$ by linear transformations of the ambient sheaf $\CO_{\BP^1}^n$, while $\BC^*$ acts on $\CM_{\gamma}$ by multiplying the base $\BP^1$ of the sheaves. The resulting $T\times \BC^*$ action on $\CM_\gamma$ will give us equivariant cohomology classes. \\

We will study the generating function
\begin{equation}\label{generating function}
Z(m)=\sum_{\gamma \in Q^-} e^{\gamma}\int_{\CM_{\gamma}} c(\CT\CM_{\gamma},mx)
\end{equation}
In the above, $c(\CT\CM_{\gamma},\cdot)$ denotes the $T\times \BC^*-$equivariant Chern polynomial of the tangent bundle of $\CM_\gamma$, $m\in \BC$ is a parameter and $x$ is the standard coordinate on $\BC$. This generating function was first introduced by Nekrasov in ~\cite{N1}, in the more complicated setting when $\fsl_n$ is replaced by its affine counterpart $\widehat{\fsl_n}$. \\

In our $\fsl_n$ situation, Braverman (\cite{B2}) has conjectured that $Z(m)$ is very closely related to the \emph{quantum trigonometric Calogero-Sutherland integrable system}. In fact, in ~\cite{B1} he proves the $m\rightarrow \infty$ limit case of this conjecture. Explicitly, Braverman considers the generating function
\begin{equation} \label{function Z}
Z=\sum_{\gamma\in Q^-} e^{\gamma}\int_{\CM_{\gamma}} 1
\end{equation}
and shows that $Z$ is the eigenfunction of the \emph{quantum
Toda hamiltonian} (up to a constant factor). We will prove Braverman's conjecture for general $m$, namely that $Z(m)$ equals the eigenfunction of the quantum trigonometric Calogero-Sutherland hamiltonian, up to a factor which we will specify. \\

The function $Z(m)$ is expected to be closely related to the equivariant $J-$function of the quantum cohomology of the cotangent bundle of the complete flag variety of $\BC^n$. The affine analogue of $Z(m)$ appears in $\CN=2$ super-symmetric $4-$dimensional gauge theory with adjoint matter, and is conjecturally very closely related to the \emph{quantum elliptic Calogero-Moser integrable system}. Both of these directions are described in ~\cite{B2}. \\

Andrei Okounkov has suggested that one can study $Z(m)$ by modeling it
as the character of a certain geometric operator $A(m)$. In this paper, we follow this
idea by relating $A(m)$ to $\fsl_n-$intertwiners. Then we use the theory of
generalized characters developed by Etingof, Frenkel and Kirillov to compute $Z(m)$. In this way we obtain:
\begin{equation} \label{main result}
Z(m)=Y_{\frac ax,m}\cdot e^{-\frac {a}{x}}\cdot \left(\prod_{\alpha\in R^+}
\frac 1{1-e^{-\alpha}}\right)^{m+1}
\end{equation}
where $a,x$ are standard coordinates on $T$ and $\BC^*$ (respectively). In the above, $Y_{\frac ax,m}$ is the eigenfunction of the Calogero-Sutherland
hamiltonian with eigenvalue $\frac {(a,a)}{x^2}$ and highest term $e^{\frac ax}$. \\

Let us say a few words about the structure of the paper: in Section~\ref{not
and basic} we introduce certain notations and concepts pertaining to Lie algebras, Verma modules, differential operators
and quiver representations. In Section~\ref{equiv cohomology} we give a short presentation of equivariant
cohomology with respect to the action of a compact Lie group. In Section~\ref{Quasiflag
Spaces} we consider the moduli spaces $\CM_{\gamma}$ in detail and establish some of
their properties. In Section~\ref{the vector bundle E}, we introduce a very important vector bundle $E$ and give several interpretations of it. In Section~\ref{the operators A(m)} we use the bundle $E$ to construct the operator $A(m)$ and relate its character to $Z(m)$. In Section~\ref{traces} we use the theory of generalized characters of intertwiners (see, for example, ~\cite{EK}) to compute the character of
$A(m)$. We prove relation \eqref{main result} in Theorem~\ref{final computation character}. Finally, as a corollary, we take the limit as $m\rightarrow \infty$ and obtain the finite-dimensional statement of Corollary 3.7 of ~\cite{B1}. \\

I would first and foremost like to thank Andrei Okounkov, my thesis advisor for this paper, for suggesting this wonderful problem and for all his patient help and advice along the way. Many of the ideas in this paper either belong to him or were inspired by his perspective. I would also like to thank Alexander Braverman and Michael Finkelberg for numerous discussions and explanations, many of which developed even beyond the scope of this paper. My gratitude also goes to Valerio Toledano Laredo for helping me understand some aspects of Ringel-Hall algebras, and to Rahul Pandharipande for suggesting solutions to some issues in Section~\ref{the vector bundle E}. \\

\section{Basic Definitions}
\label{not and basic}

\subsection{}
\label{lie algebra}

In this section we will describe some aspects of the representation theory of the simple Lie algebra $\fsl_n$. Fix a basis $w_1,...,w_n$ of $\BC^n$. To this choice of basis, there corresponds a Cartan subalgebra $\fh\subset \fsl_n$ and a root system $R\subset \fh^*$. Corresponding to $R$ we have the standard decomposition
$$
\fsl_n=\fn^-\oplus \fh \oplus \fn^+
$$
Explicitly, $\fsl_n$ consists of traceless $n\times n$ complex matrices, $\fh$ consists of traceless diagonal matrices and $\fn^-$/$\fn^+$ consists of strictly lower/upper triangular $n\times n$ matrices (with respect to the basis $w_1,...,w_n$). \\

The root system $R$ consists of vectors $w_i-w_j$ ($1\leq i\neq j\leq n$). If $i<j$ the corresponding root is called positive, while if $i>j$ it is called negative. The set of positive roots is denoted by $R^+$, while the set of negative roots is denoted by $R^-$. The half-sum of the positive roots will be denoted by
$$
\rho=\frac 12\sum_{\alpha \in R^+} \alpha
$$
Further, $\alpha_i=w_i-w_{i+1}\in R^+$ will be called the simple roots. Inside $\fh^*$ we have the root lattice
$$
Q=\bigoplus_{i=1}^{n-1} \BZ \alpha_i
$$
Letting $\BZ_-=\{0,-1,-2,...\}$, the lattice $Q$ contains the monoid
$$
Q^-=\bigoplus_{i=1}^{n-1} \BZ_- \alpha_i
$$
One can impose a partial order on weights $\lambda\in \fh^*$ by setting $\lambda\leq \mu$ if $\lambda-\mu\in Q^-$. \\

The vector spaces $\fn^-, \fh, \fn^+$ have standard bases
$$
\fn^-=\langle f_{\alpha}, \alpha\in R^+\rangle
$$
$$
\fh=\langle h_1,...,h_{n-1}\rangle
$$
$$
\fn^+=\langle e_{\alpha}, \alpha\in R^+\rangle
$$
Explicitly, if $\alpha=w_i-w_j$ with $i<j$, then $f_{\alpha}$ is the matrix with 1 at the intersection of row $j$ and column $i$ and 0 everywhere else, while $e_{\alpha}$ is the matrix with 1 at the intersection of row $i$ and column $j$ and 0 everywhere else. When $\alpha=w_i-w_{i+1}$ is a simple root, we write $f_i=f_{\alpha_i}$ and $e_i=e_{\alpha_i}$. Note that the $f_i$'s (respectively, the $e_i$'s) generate $\fn^-$ (respectively, $\fn^+$) as Lie algebras. Finally, one defines
$$
h_i=[e_i,f_i].
$$
\textrm{ }

\subsection{}
\label{def Verma}

Given $\beta\in \fh^*$ and a set of linearly independent vectors
$\beta_1,...,\beta_t\in \fh^*$, we define a \emph{cone} to be the set
\begin{equation}
C=\{\lambda|\lambda=\beta+k_1\beta_1+...+k_t\beta_t, \textrm{ } k_1,...,k_t\in
\BN_0\}\subset \fh^*
\end{equation}
An example of a cone is $Q^-$ defined in the previous subsection. An $\fsl_n-$module $V$ is said to posses a \emph{conic weight
decomposition} if for some cone $C\subset \fh^*$, one has
\begin{equation} \label{weight decomposition}
V=\bigoplus_{\lambda \in C} V[\lambda]
\end{equation}
where each $V[\lambda]$ is a finite-dimensional subspace such that Cartan elements $h\in \fh$ act on it as multiplication by $\langle h,\lambda \rangle$.  The most important example of such a module is the Verma module.

\begin{definition} We define the Verma module $M(\lambda_0)$ with highest weight $\lambda_0$ to be the $\fsl_n-$module freely generated over $U\fn^-$ by a single vector $v$ under the relations
$$e\cdot v=0, \textrm{ }\forall e \in \fn^+ $$
$$h\cdot v=\langle h,\lambda_0\rangle v, \textrm{ }\forall h\in \fh$$
\end{definition}
The representation $M(\lambda_0)$ is irreducible for generic
$\lambda_0$, and it has a conic weight decomposition with highest
weight $\lambda_0$:
$$
M(\lambda_0)=\bigoplus_{\lambda\in \lambda_0+Q^-} M(\lambda_0)[\lambda].
$$
\textrm{ }

\subsection{}
\label{def generalized character}

Consider two $\fsl_n-$modules $U$ and $V$, where $V$ has a conic weight
decomposition as in \eqref{weight decomposition}. Given a linear
operator $\Phi:V\rightarrow V\otimes U$, we define its character as
the formal $U-$valued expression
\begin{equation} \label{definition character}
\chi_{\Phi}=\sum_{\lambda \in C} e^{\lambda}\cdot
\textrm{Tr}(\Phi|_{V[\lambda]})
\end{equation}
Here the $e^{\lambda}$ should be perceived as formal commuting symbols (with the standard relation $e^{\lambda_1+\lambda_2}=e^{\lambda_1}e^{\lambda_2}$). The operator $\Phi|_{V[\lambda]}$ is defined as the restriction of $\Phi$ to the finite-dimensional factor $V[\lambda]$, followed by projection onto $V[\lambda]\otimes U$. \\

As noted in ~\cite{EK}, the character can also be perceived as the equivariant function on the Lie group $SL_n$ given by
$$
\chi_{\Phi}(g)=\textrm{Tr}(\Phi\cdot g)
$$
Here, equivariant means that for $x,g\in SL_n$ we have
\begin{equation} \label{equiv}
\chi_{\Phi}(xgx^{-1})=x\cdot \chi_{\Phi}(g)
\end{equation}
The equivariance implies that $\chi$ is determined by its values on elements $e^h$ of the maximal torus. The two descriptions of $\chi_{\Phi}$ presented above are related by the fact that
$$
\chi_{\Phi}(e^h)=\sum_{\lambda \in C} e^{\langle h, \lambda
\rangle}\cdot \textrm{Tr}(\Phi|_{V[\lambda]})
$$
The above sum converges for $h$ in a certain cone in $\fh$. \\

\subsection{}
\label{differential operators}

The character defined in \eqref{definition character} is a particular example of a \emph{power series}. We define a power series to be an expression of the form
\begin{equation}\label{power series}
\chi=\sum_{\lambda\in C} e^{\lambda}c_{\lambda}
\end{equation}
where $C\subset \fh^*$ is a cone, and $c_{\lambda}$ are coefficients belonging
to some representation $U$ of $\fsl_n$. A particularly important power series
is the \emph{Weyl denominator}
\begin{equation} \label{Weyl denominator}
\delta=\prod_{\alpha \in R^+} (e^{\alpha/2}-e^{-\alpha/2})=e^{\rho}\prod_{\alpha \in R^+} (1-e^{-\alpha})
\end{equation}
and its inverse
$$
\delta^{-1}=e^{-\rho}\prod_{\alpha \in R^+} \frac 1{1-e^{-\alpha}}
$$
Both $\delta$ and $\delta^{-1}$ are power series with $\pm \rho+Q^-$ as the associated cones. In general, two types of objects can act on power series $\chi$: elements of $\fsl_n$ can act on the coefficients $c_{\lambda}$, and \emph{differential operators} can act on the symbols $e^{\lambda}$. Examples of such differential operators are the partial derivative in the direction of $\alpha \in \fh^*$:
\begin{equation}
\partial_{\alpha}(e^{\lambda})=(\alpha, \lambda)e^{\lambda}
\end{equation}
and the Laplace operator:
\begin{equation}
\Delta_{\fh}(e^{\lambda})=(\lambda, \lambda)e^{\lambda}.
\end{equation} \\

The main differential operator we
will be concerned with is the \emph{quantum trigonometric Calogero-Sutherland
hamiltonian} (\cite{EK}):
$$
L(m)=\Delta_{\fh}-2m(m+1)\sum_{\alpha\in R^+} \frac 1{(e^{\alpha/2}-e^{-\alpha/2})^2}
$$
Consider a power series
\begin{equation} \label{eigenfunction CS}
Y_{\lambda_0,m}=\sum_{\lambda \in \lambda_0+Q^-} e^{\lambda}c_{\lambda}
\end{equation}
normalized such that $c_{\lambda_0}=1$. The fact that $Y_{\lambda_0,m}$ is an eigenfunction of $L(m)$ with eigenvalue $s\in \BC^*$ is equivalent to the following relations on the coefficients $c_\lambda$:
\begin{equation}
c_{\lambda}\cdot ((\lambda, \lambda)-s)=2m(m+1)\sum_{\alpha\in
R^+}\sum_{j\geq 1} j \cdot c_{\lambda+j\alpha}
\end{equation}
The above expression for $\lambda=\lambda_0$ implies $s=(\lambda_0,\lambda_0)$. Furthermore, for generic $\lambda_0\in \fh^*$, the above relation recursively determines all coefficients $c_{\lambda}$ from $c_{\lambda_0}=1$. Thus there is a unique eigenfunction of $L(m)$ with highest term $e^{\lambda_0}$, and we will henceforth denote it by $Y_{\lambda_0,m}$. \\

\subsection{}
\label{quivers}

We need to touch on one more topic from representation theory, which is quite independent of what was discussed above. A \emph{quiver} $Q$ with vertex set $I$ and edge set $E$ is a directed graph. A representation $V$ of $Q$ will be a collection of vector spaces $V_i$ for every $i\in I$, together with linear maps $\phi_e:V_i\rightarrow V_j$ for every edge $e\in E$ between vertices $i$ and $j$. The dimension of a representation $V$ will be the vector of positive integers $\textrm{dim}(V)=(d_i, i\in I)$, where $d_i$ is the dimension of the vector space $V_i$. A map of two representations $V$ and $V'$ is a collection of maps between the vector spaces $V_i$ and $V_i'$ (as $i$ ranges over $I$) that commute with the edge maps. \\

A simple representation is one that has no non-trivial subrepresentations. An indecomposable representation is one which does not decompose non-trivially into a direct sum. Let $S$, $R$ and $\textrm{Rep}(Q)$ denote the sets of isomorphism classes of simple, indecomposable and all representations of the quiver $Q$, respectively. Any representation is a direct sum of indecomposables in $R$, and thus $\textrm{Rep}(Q)=\BN[R]$. \\

Define the \emph{Ringel-Hall algebra} $U_q(Q)$ as the algebra generated by symbols $e_{\kappa}$ for $\kappa\in \textrm{Rep}(Q)$, with multiplication given by the rule
$$
e_{\kappa'}\cdot e_{\kappa''}=q^{\langle \textrm{dim}(\kappa'),\textrm{dim}(\kappa'')\rangle}\sum_{\kappa\in \textrm{Rep}(Q)} P_{\kappa',\kappa''}^{\kappa}(q) e_{\kappa}
$$
where $\langle \cdot, \cdot \rangle$ is the Euler form of the quiver. The definition of $P_{\kappa',\kappa''}^{\kappa}$ is the following: suppose the ground field is $\BF_q$, the finite field with $q$ elements. Then take a representation $V$ from the class $\kappa$, and define $P_{\kappa',\kappa''}^{\kappa}(q)$ to be the number of subrepresentations $V' \subset V$ such that $V'$ is in the isomorphism class $\kappa'$ and $V/V'$ is in the isomorphism class $\kappa''$. The function $P_{\kappa',\kappa''}^{\kappa}(q)$ is a polynomial in $q$ (see ~\cite{R1}), so the definition of the multiplication law makes sense for $q$ an indeterminate. In particular, we can specialize the multiplication law at $q=1$ and obtain the algebra $U(Q):=U_1(Q)$. \\

In this paper, we will only be concerned with the quiver
$$
A_{n-1}:1\rightarrow 2\rightarrow...\rightarrow n-1
$$
with vertex set $I=\{1,2,...,n-1\}$. Its simple representations are denoted by $[i;1)$, where $i\in I$. By definition, $[i;1)$ is the representation with a one-dimensional vector space at the vertex $i$ and all edge maps 0. The indecomposable representations are denoted by $[i;l)$, where $i\in I$ and $1\leq l\leq n-i$. By definition, $[i;l)$ is the representation with a one-dimensional vector space at the vertices $i,i+1,...,i+l-1$ and maps
$$
...@>>>0@>>>V_i@>{\cong}>> V_{i+1}@>{\cong}>>...@>{\cong}>> V_{i+l-1} @>>> 0@>>>...
$$
One can explicitly compute the Ringel-Hall algebra $U(A_{n-1})$. If we denote $e_i:=e_{[i;1)}$, we have
$$
e_{[i;l)}=[e_{i+l-1}, [e_{i+l-2},[....[e_{i+1},e_i]...]]]
$$
where $[\cdot, \cdot]$ is the standard commutator. For a general isomorphism class $\kappa\in \textrm{Rep}(A_{n-1})$, we have
\begin{equation} \label{formula quiver}
\kappa=\dsp \bigoplus_{i=1}^{n-1} \bigoplus_{l=1}^{n-l} [i;l)^{\oplus k_{il}} \Rightarrow
e_{\kappa}=\prod_{i=1}^{n-1} \prod_{l=1}^{n-l} \frac {e_{[i;l)}^{k_{il}}}{k_{il}!}
\end{equation}
In particular, the above shows us that $U(A_{n-1})$ is generated by $e_1,...,e_{n-1}$. Let $E_{i,j}$ denote the matrix with 1 at the intersection of row $i$ and column $j$, and 0 everywhere else. Then we have the following (\cite{R1}):

\begin{theorem} \label{Ringel-Hall = enveloping}
The assignment $e_i\rightarrow E_{i+1,i}$ gives us an isomorphism
$$U(A_{n-1}) @>{\cong}>> U(\fn^-)$$
where $U(\fn^-)$ is the universal enveloping algebra of $\fn^-\subset \fsl_n$. \\
\end{theorem}

\section{Equivariant Cohomology}
\label{equiv cohomology}

\subsection{} Our presentation of equivariant cohomology will follow ~\cite{AB}. Let us consider a Lie group $G$, let $BG$ be the classifying space of $G$ and $\pi:EG\rightarrow BG$ the universal $G-$bundle. Given a smooth variety $X$ with a $G-$action, define the space
$$
EG\times_G X=EG\times X / \{(pg,x)\sim (p,gx), p\in EG, x\in X, g\in G\}
$$
The $G-$equivariant cohomology of $X$ is defined as
$$
H^*_G(X):=H^*(EG\times_G X)
$$
When the group $G$ is clear from context, we will simply call this construction \emph{equivariant cohomology}. \\

\subsection{ }
\label{properties of equivariant cohomology}

Equivariant cohomology satisfies the following properties:
\begin{itemize}
  \item \emph{existence of pull-backs: } if $f:X\rightarrow Y$ is a
  $G-$equivariant map, then there exists a pull-back homomorphism $f^*:H^*_G(Y)\rightarrow H^*_G(X)$.
  \item \emph{module structure: } if we consider $\pi:X\rightarrow pt$, then $\pi^*:H^*_G(pt)\rightarrow H^*_G(X)$ endows $H^*_G(X)$ with a structure of a $H^*_G(pt)-$module.
  \item \emph{cohomology of the point: } $H^*_G(pt)=\BC[x_1,...,x_n]$, where $x_1,...,x_n$ are coordinates of the Lie algebra of a maximal torus of $G$.
\end{itemize}
Moreover, if $f:X\rightarrow Y$ is a proper $G-$equivariant map, we can define push-forward maps $f_*:H^*_G(X)\rightarrow H^*_G(Y)$ (which are homomorphisms of $H^*_G(pt)-$modules) with the following properties:
\begin{itemize}
  \item \emph{projection formula: }
  \begin{equation} \label{equiv cohom 1}
  f_*(c\cdot f^*d)=f_*c\cdot d
  \end{equation}
  \item \emph{base change: } Suppose we have a fiber square
$$
\begin{CD}
Y\times_T Z @>{p}>> Y \\
@V{q}VV @V{r}VV \\
Z @>{s}>> T
\end{CD}
$$
such that the push-forward maps $r_*$ and $q_*$ are defined. If this happens, then we have the equality
\begin{equation} \label{equiv cohom 2}
s^*r_*=q_*p^*
\end{equation}
  \item \emph{push-forward of inclusions: } If $i:X\hookrightarrow Y$ is an inclusion, then $i_*1=[X]$, where $[X]$ is the Poincare dual of the class of the subvariety $X$ in $Y$. This and \eqref{equiv cohom 1} imply that
  \begin{equation} \label{equiv cohom 3}
  i_*i^*c=c\cdot [X]
  \end{equation}
for any $c\in H_G^*(Y)$. Moreover, for any $c\in H_G^*(X)$, ~\cite{AB} tells us that
  \begin{equation} \label{equiv cohom 4}
  i^*i_*c=c\cdot e(N_{Y|X})
  \end{equation}
  In the above, $e(N_{Y|X})$ denotes the top Chern class of the normal bundle of $X$ in $Y$.

\end{itemize}
When $\pi:X\rightarrow pt$ is the projection to a point, then the push-forward is just the integral:
$$
\pi_*\alpha=\int_X\alpha
$$\\

\subsection{}
\label{fixed points cohom}

We will be interested in the case when $G=T=(\BC^*)^k$ is a torus,
and the fixed locus $X^T$ of the $T-$action on $X$ is a finite union
of fixed points. In that case, we have the following
\emph{localization theorem}:

\begin{theorem} \label{locaization thm}
The restriction map
$$
H:=H^*_T(X)\bigotimes_{H^*_T(pt)} \emph{Frac}(H^*_T(pt)) @>{\cong}>> H^*_T(X^T)\bigotimes_{H^*_T(pt)} \emph{Frac}(H^*_T(pt))
$$
is an isomorphism. For $p\in X^T$, let $i_p:pt\rightarrow X$ denote the inclusion of $p$ in $X$. Then the classes
$$
[p]:={i_p}_*1
$$
form a basis of $H$ as a vector space over $\emph{Frac}(H^*_T(pt))$.
\end{theorem}

This basis is very convenient, because $i_p^*[q]=0$ for $p\neq q$. Moreover, by \eqref{equiv cohom 4} we have
\begin{equation} \label{restriction to a point}
i_p^*[p]=e(\CT_p X) =\prod_{w\in \CT_p X} w
\end{equation}
where $w\in \CT_pX$ are the weights of the $T-$action in
the tangent space at $p$. This means that if we want to compute a
certain class $c\in H^*_T(X)$ in terms of the basis vectors $[p]$, all we
have to do is to compute its restrictions to the fixed points. In other
words, we have the following \emph{localization formula}:

\begin{theorem} \label{localization formula}
For any $c\in H^*_T(X)$ we have
$$
c=\sum_{p\in X^T} [p]\cdot \frac {i_p^*c}{\dsp \prod_{w\in \CT_p X} w}
$$
\end{theorem}

Applying $\pi_*$ to the above formula, where $\pi:X\rightarrow pt$,
gives us the \emph{integral formula}:

\begin{corollary} \label{integral formula}
For any $c\in H^*_T(X)$ we have
$$
\int_X c=\pi_*c=\sum_{p\in X^T} \frac {i_p^*c}{\dsp \prod_{w\in \CT_p X} w}
$$
\end{corollary}

We have already noted above the importance of the top Chern class $e$. In
fact, we also have the equivariant version of
Proposition 12.8 in ~\cite{BT}:

\begin{proposition} \label{class of the section}
Suppose $V$ is a $T-$vector bundle on $X$ which possesses a regular section with zero locus $Z\subset X$. If the decomposition of $Z$ into irreducible components is $Z=\cup_i Z_i$, then
$$
e(V)=\sum_i [Z_i].
$$ \\
\end{proposition}

\section{Laumon Quasiflag Spaces}
\label{Quasiflag Spaces}

\subsection{}
\label{Def quasiflags}

Recall that we have chosen a basis $(w_1,...,w_n)$ of $\BC^n$. Let $T\subset SL_n$ be the maximal torus of matrices which are diagonal in this basis. Let $\ft$ be the Lie algebra of $T$, and let $\fh=\ft^*$ be its dual. Whenever we will mention roots and weights from now on, we will always refer to the roots and weights of $\fh$ (as in Section~\ref{not and basic}). For any $\gamma=-d_1\alpha_1-...-d_{n-1}\alpha_{n-1}\in Q^-$, let $\CM_{\gamma}$ denote the moduli space of Laumon quasiflags
\begin{equation} \label{flag of sheaves}
\CF:0\subset \CF_1\subset...\subset \CF_{n-1}\subset \CO^n
\end{equation}
In the above, $\CO$ stands for $\CO_{\BP^1}$, and each $\CF_i$ is a torsion-free sheaf on $\BP^1$ of rank
$i$ and degree $-d_i$, such that $\CF_i|_\infty=\textrm{span}(w_1,...,w_i)$. Fixing the behavior of our flags at
$\infty$ is called \emph{framing}. As in ~\cite{BF}, the spaces $\CM_{\gamma}$ are smooth and of
dimension $2(d_1+...+d_{n-1})$. We can define the disjoint union:
\begin{equation}
\CM=\bigsqcup_{\gamma\in Q^-} \CM_{\gamma}.
\end{equation}

\subsection{} The group $T\times \BC^*$ acts on $\CM_{\gamma}$ in the following way. The torus $T$ acts
on the ambient sheaf $\CO^n$ by changes of basis and $\BC^*$ acts on the sheaves by multiplying the base $\BP^1$. The fixed points of $\CM_{\gamma}$ under this torus action are precisely flags which have
\begin{equation} \label{fixed flag}
\CF_i=w_1\cdot \CO(-d_i^1)\oplus...\oplus w_i\cdot \CO(-d_i^i)
\end{equation}
where $\CO(-1)$ will always denote $\CO_{\BP^1}(-1\cdot 0_{\BP^1})$. Thus a fixed flag is determined by a
vector of non-negative integers
\begin{equation}
\label{fixed flag vector of integers}
d=(d_j^i), \textrm{ }\textrm{ }1\leq i\leq j\leq n-1
\end{equation}
satisfying
\begin{equation} \label{flag existence condition}
d_j^i\geq d_{j+1}^i, \textrm{
}\textrm{ }d_j^1+d_j^2+...+d_j^j=d_j, \textrm{ }\textrm{ }\forall \textrm{ }1\leq i\leq j\leq n-1
\end{equation}
We will always write $d_n^i=0$ by convention. Since the vector $d$ completely determines the fixed flag \eqref{fixed flag}, we will often abuse notation and write $d$ for the flag as well. \\

\subsection{}
\label{equiv cohom quasiflags}

Let $a_1,...,a_n$ be standard coordinates on $\BC^n$, and $x$ be the standard coordinate on $\BC$. Then $a_1-a_2,...,a_{n-1}-a_n,x$ will be coordinates on $\textrm{Lie}(T\times \BC^*)=\ft\oplus \BC$. As mentioned in Section~\ref{properties of equivariant cohomology}, in this case we have
\begin{equation}
H^*_{T\times \BC^*}(pt)=\BC[a_1-a_2,...,a_{n-1}-a_n,x]
\end{equation}
Consider
\begin{equation} \label{def of H}
H:=\bigoplus_{\gamma \in Q^-} H^*_{T\times \BC^*}(\CM_\gamma) \bigotimes_{H^*_{T\times \BC^*}(pt)} \textrm{Frac}(H^*_{T\times \BC^*}(pt))
\end{equation}
By the description of the torus fixed points in \eqref{fixed flag}, one sees that each $\CM_{\gamma}$ contains finitely many fixed points. Therefore the localization Theorem~\ref{locaization thm} states that $H$ is a vector space over $\textrm{Frac}(H^*_{T\times \BC^*}(pt))$, with a basis given by the cohomology classes of
the torus fixed points. Since these torus fixed points are indexed by vectors $d=(d_j^i)$ as in \eqref{fixed flag vector of integers}, we will denote their
classes in $H$ by the symbols $[d]$. \\

\subsection{}
\label{action on H}

In ~\cite{FK}, Finkelberg and Kuznetsov introduced a geometric structure of an $\fsl_n-$module on $H$, which we will describe below.
Given two flags $\CF,\CF'\in \CM$, we will write $\CF'\subset_i \CF$ if
$\CF'_j=\CF_j$ for $j\neq i$, while $\CF'_i\subset \CF_i$ in such a
way that $\textrm{length}(\CF_i/\CF'_i)=1$. Then for $1\leq
i\leq n-1$, let
\begin{equation}
\fC_i=\{(\CF, \CF')\in \CM \times \CM|\CF' \subset_i \CF\}
\end{equation}
As in ~\cite{BF}, we will call $\fC_i$ a \emph{simple
correspondence}. It is a smooth subvariety of $\CM \times \CM$ which will have middle dimension $2(d_1+...+d_{n-1})+1$ inside the component $\CM_{\gamma}\times \CM_{\gamma-\alpha_i}$. \\

Let $p,q:\CM \times \CM\rightarrow \CM$ be the standard projections onto the first and second factors. Define the operators $e_i,f_i:H \rightarrow H$ by
$$
e_i(\alpha)= p_*([\fC_i]\cdot q^*\alpha)
$$
$$
f_i(\alpha)= - q_*([\fC_i]\cdot p^*\alpha)
$$
If we let $\tilde{p},\tilde{q}:\fC_i\rightarrow \CM$ be the projections restricted to $\fC_i$, then the above definitions are equivalent to $e_i=\tilde{p}_*\tilde{q}^*$ and $f_i=-\tilde{q}_*\tilde{p}^*$. Theorem 3.4 in ~\cite{B1} gives us the following result:

\begin{theorem} \label{cohomology is Verma}
The operators $e_i$ and $f_i$ generate a well-defined $\fsl_n-$action on $H$. Under this action, $H$ is isomorphic to the Verma module $M(\frac ax-\rho)$.
\end{theorem}

\begin{remark}
The slight difference between the above theorem and its counterpart in
~\cite{B1} is due to the fact that we use highest weight Verma modules instead of lowest weight Verma modules, but the theory is analogous. \\
\end{remark}

Let $\gamma=-d_1\alpha_1-...-d_{n-1}\alpha_{n-1}$, and consider the cohomology class $[d]\in H$ corresponding to a fixed point $d\in \CM_\gamma^{T\times \BC^*}$. We want to write down how the Cartan elements $h_i=[e_i,f_i]$ act on $[d]$. Section 3.7 of ~\cite{FK} tells us that
\begin{equation} \label{action of h}
h_i([d])=\left(\frac {a_i}x-\frac {a_{i+1}}x-1+d_{i-1}-2d_i+d_{i+1}\right)\cdot [d]=\left\langle h_i, \frac ax-\rho+\gamma\right\rangle \cdot [d]
\end{equation}
Therefore the weight spaces of the $\fsl_n-$module $H$ are precisely the summands of \eqref{def of H}, namely
\begin{equation} \label{weights}
H\left[\frac ax-\rho+\gamma\right]=H^*_{T\times \BC^*}(\CM_\gamma) \bigotimes_{H^*_{T\times \BC^*}(pt)}\textrm{Frac}(H^*_{T\times \BC^*}(pt)).
\end{equation} \\

\subsection{}
\label{general correspondences}

One can generalize the definition of the simple correspondences $\fC_i$ to obtain even more operators on $H$, with the goal of studying the locus
\begin{equation}
\fC=\{(\CF, \CF')\in \CM \times \CM|\CF' \subset \CF\}
\end{equation}
The notation $\CF'\subset \CF$ means that $\CF'_i\subset \CF_i$ for all $1\leq i\leq n-1$, without any extra conditions. For such $\CF'\subset \CF$, the ``quotient'' $\CF/\CF'$ should be interpreted as the flag of quotients
\begin{equation} \label{flag of quotients}
\CF_1/\CF_{1}'\rightarrow...\rightarrow \CF_{n-1}/\CF_{n-1}'
\end{equation}
In the above, the maps are induced by the flag inclusion maps. Because $\CF_i$ and $\CF'_i$ have the same rank for all $i$, the quotient $\CF/\CF'$ is supported at finitely many points:
$$
\textrm{supp}(\CF/\CF')=\sum_{s\in \BP^1} \alpha_s\cdot s
$$
For each $s\in \BP^1$ that appears in the above sum, $\alpha_s$ is a vector of natural numbers $(\alpha_s^{(1)},...,\alpha_s^{(n-1)})$ such that $\alpha_s^{(i)}$ is the length of the quotient sheaf $\CF_i/\CF'_i$ at $s$. \\

The stalk of $\CF/\CF'$ at any given point $s\in \BP^1$ is a flag of finite dimensional vector spaces (with maps given by \eqref{flag of quotients}), and is thus a representation of the quiver $A_{n-1}$ of Section~\ref{quivers}. Let
$$
\gamma_s\in \textrm{Rep}(A_{n-1})
$$
denote the isomorphism class of this representation. We know from Section~\ref{quivers} that any isomorphism class $\kappa\in \textrm{Rep}(A_{n-1})$ can be uniquely written as a direct sum $\kappa=\oplus_{j=1}^m \theta_j$, where the $\theta_j$ are classes of indecomposable representations. Then we define the locus
$$
\fC_{\kappa}^{\circ} \subset \fC
$$
to be the set of $(\CF,\CF')$ with $\CF'\subset \CF$ which satisfy the conditions:
\begin{enumerate}
  \item there exist distinct points $s_1,...,s_m\in \BP^1$ such that
$$
\textrm{supp}(\CF/\CF')=\sum_{j=1}^m \textrm{dim}(\theta_j)\cdot s_j
$$
  \item $$\gamma_{s_j}=\theta_j$$
\end{enumerate}
Finally, we will call the closure $\fC_{\kappa}=\overline{\fC_{\kappa}^{\circ}}\subset \fC$ a \emph{correspondence}. Note that when $\kappa$ is the simple representation corresponding to the vertex $i$ of the quiver, then $\fC_{\kappa}=\fC_i$ as defined in the previous subsection. In ~\cite{FK} we encounter the following result:

\begin{proposition} \label{irreducible components}
The correspondences $\fC_{\kappa}$ are precisely the distinct irreducible components of $\fC$ of maximal dimension, and all have dimension equal to half the dimension of $\CM\times \CM$.
\end{proposition}

Since $\CM\times \CM$ is a disjoint union of components of various dimensions, the above statement needs a few clarifications. The correspondence $\fC_\kappa$ intersects the component $\CM_\gamma \times \CM_{\gamma'}$ only if $\gamma'-\gamma=\textrm{dim}(\kappa)$. What the proposition says is that in each component $\CM_\gamma\times \CM_{\gamma'}$, the subvarieties $\fC_\kappa$ with $\gamma'-\gamma=\textrm{dim}(\kappa)$ are precisely the distinct irreducible components of $\fC$ of maximal dimension. Moreover, the dimension of $\fC_\kappa$ in the component $\CM_\gamma\times \CM_{\gamma'}$ is equal to half the dimension of $\CM_\gamma\times \CM_{\gamma'}$. \\

The correspondences $\fC_{\kappa}$ define operators on $H$ via
$$
e_{\kappa}(\alpha)= p_*([\fC_\kappa]\cdot q^*\alpha)
$$
where $p,q:\CM\times \CM\rightarrow \CM$ are the standard projections onto the first and second factors, respectively. To understand how the operators $e_{\kappa}$ act, we have the following:

\begin{theorem} \label{operators Ringel-Hall}
The map that sends the elements $e_{\kappa}\in U(A_{n-1})^{op}$ to the operators $e_{\kappa}$ defined above is an algebra homomorphism. Here $U(A_{n-1})^{op}$ means the algebra opposite to the Ringel-Hall algebra of $A_{n-1}$.
\end{theorem}

In other words, the operators $e_{\kappa}$ follow the same multiplication rules as the corresponding symbols from the opposite Ringel-Hall algebra of $A_{n-1}$. We will not prove Theorem~\ref{operators Ringel-Hall} here, since it has been proved in the more general case of parabolic flags of sheaves in ~\cite{GFK}, Theorem 7.12. Note that in ~\cite{GFK} the theorem is proved for the operators given by the transposed correspondences $\fC_{\kappa}^T$. Though the proof is completely analogous, this accounts for the word ``opposite'' in the statement of Theorem~\ref{operators Ringel-Hall}. \\

\section{The Vector Bundle $E$}
\label{the vector bundle E}

\subsection{} Let us consider the sheaf $E$ on $\CM\times \CM$, whose fiber over $(\CF, \CF')\in \CM \times \CM$ is
\begin{equation} \label{definition of E}
E|_{(\CF, \CF')}=\textrm{Hom}(\CF'(\infty), \CO^n/\CF)
\end{equation}
By definition, elements of $\textrm{Hom}(\CF'(\infty), \CO^n/\CF)$ are tuples $(\phi_1,...,\phi_{n-1})$ of morphisms $\phi_i:\CF'_i(\infty)\rightarrow \CO^n/\CF_i$ such that the diagram

\begin{equation} \label{n-1 tuple of homs}
\begin{CD}
\CF'_1(\infty) @>{\phi_1}>> \CO^n/\CF_1 \\
@VVV @VVV \\
\CF'_2(\infty) @>{\phi_2}>> \CO^n/\CF_2 \\
@VVV @VVV \\
... @>>> ... \\
@VVV @VVV \\
\CF'_{n-1}(\infty) @>{\phi_{n-1}}>> \CO^n/\CF_{n-1} \\
\end{CD}
\end{equation}
commutes. In the above, all the vertical maps are induced by the flag inclusions.

\begin{proposition} \label{E tangent bundle}
The restriction of $E$ to the diagonal $\Delta \subset \CM\times
\CM$ is just the tangent bundle of $\CM$.
\end{proposition}

\begin{proof} The proposition is merely a consequence of the fact that the tangent space to $\CM$ at $\CF$ is $\textrm{Hom}(\CF(\infty), \CO^n/\CF)$, as described in ~\cite{BE} and ~\cite{CF}. The extra torsion by $\infty\in \BP^1$ is a consequence of the fact that our flags are fixed at that point, and thus there cannot be any deformation of their fibers at $\infty\in \BP^1$.
\end{proof} \\

\subsection{} The sheaf $E$ admits an alternative description, which will be very useful in establishing certain technical results later on. Let $\pi:\BP^1\times \BA^1\rightarrow \BP^1$ be the standard
projection, and denote $L_0=\BP^1\times 0$. Following~\cite{BI}, to any flag of sheaves
as in \eqref{flag of sheaves} we associate the torsion-free subsheaf
$\tilde{\CF}\subset \CO_{\BP^1\times \BA^1}^n$ given by
$$
\tilde{\CF}=\pi^*\CF_1+\pi^*\CF_2(-L_0)...+\pi^*\CF_{n-1}(-(n-2)L_0)+\CO_{\BP^1\times
\BA^1}^n(-(n-1)L_0)
$$
The sum does not refer to a direct sum, but to a sum of subsheaves of $\CO_{\BP^1\times \BA^1}^n$.
Constants $z\in \BC^*$ act on $\BP^1\times \BA^1$ by multiplying the second factor, and under this action each summand
$\pi^*\CF_{i+1}(-iL_0)$ of $\tilde{\CF}$ is preserved. Therefore, $\tilde{\CF}$ is a $z-$invariant subsheaf of $\CO^n_{\BP^1\times \BA^1}$. \\

\subsection{} Let $D_\infty=\infty\times \BA^1$. Define
$\textrm{Hom}_z(\tilde{\CF}'(D_{\infty}),\CO_{\BP^1\times
\BA^1}^n/\tilde{\CF})$ to be the space of $z-$invariant morphisms
between the sheaves in question, and we will show that there is a
natural isomorphism
\begin{equation} \label{natural iso}
\textrm{Hom}(\CF'(\infty), \CO^n/\CF)\cong
\textrm{Hom}_z(\tilde{\CF}'(D_{\infty}),\CO_{\BP^1\times
\BA^1}^n/\tilde{\CF}).
\end{equation}
To define this isomorphism, take an element $\phi\in
\textrm{Hom}(\CF'(\infty), \CO^n/\CF)$ as in \eqref{n-1 tuple of
homs}. Each one of its component morphisms
$$
\phi_{i+1}:\CF'_{i+1}(\infty)\rightarrow \CO^n/\CF_{i+1}
$$
naturally lifts to a $z-$invariant morphism
$$
\tilde{\phi}'_{i+1}:\pi^*\CF'_{i+1}(D_\infty)\rightarrow
\CO^n_{\BP^1\times \BA^1}/\pi^*\CF_{i+1}
$$
Let us twist this morphism by $-iL_0$, and then compose it with the
natural map
$$
\CO^n_{\BP^1\times
\BA^1}(-iL_0)/\pi^*\CF_{i+1}(-iL_0)\hookrightarrow
\CO^n_{\BP^1\times \BA^1}/\pi^*\CF_{i+1}(-iL_0) \twoheadrightarrow
\CO^n_{\BP^1\times \BA^1}/\tilde{\CF}
$$
In this way we obtain $z-$invariant homomorphisms
$$
\tilde{\phi}_{i+1}:\pi^*\CF'_{i+1}(D_\infty-iL_0)\rightarrow
\CO^n_{\BP^1\times \BA^1}/\tilde{\CF}
$$
By the commutativity
of the diagram \eqref{n-1 tuple of homs}, these homomorphisms paste to give a
well-defined $z-$invariant morphism $\tilde{\phi}:\tilde{\CF}'(D_{\infty})\rightarrow \CO_{\BP^1\times
\BA^1}^n/\tilde{\CF}$. \\

Conversely, take $\tilde{\phi}\in
\textrm{Hom}_z(\tilde{\CF}'(D_{\infty}),\CO_{\BP^1\times
\BA^1}^n/\tilde{\CF})$. By restriction, it gives rise to
$z-$invariant homomorphisms
$$
\tilde{\phi}_{i+1}:\pi^*\CF_{i+1}'(D_\infty-iL_0)\rightarrow
\CO_{\BP^1\times \BA^1}^n/\tilde{\CF}
$$
For each $i$, the fact that
this morphism is $z-$invariant implies that it lifts to a morphism
into $\CO_{\BP^1\times \BA^1}^n/\pi^*\CF_{i+1}(-iL_0)$, and moreover
that its image lies in $\CO_{\BP^1\times
\BA^1}^n(-iL_0)/\pi^*\CF_{i+1}(-iL_0)$. Twisting by $iL_0$ gives
rise to a $z-$invariant homomorphism
$$
\tilde{\phi}'_{i+1}:\pi^*\CF'_{i+1}(D_\infty)\rightarrow
\CO^n_{\BP^1\times \BA^1}/\pi^*\CF_{i+1}
$$
Since all the sheaves in the above are pulled back from $\BP^1$,
then Lemma~\ref{general result} below implies that $\tilde{\phi}'_{i+1}$
is the lift of a homomorphism of sheaves on $\BP^1$
$$
\phi_{i+1}:\CF'_{i+1}(\infty)\rightarrow \CO^n/\CF_{i+1}
$$
The fact that the homomorphisms $\tilde{\phi}_{i+1}$ paste to give
the homomorphism $\tilde{\phi}$ implies precisely that the morphisms
$\phi_{i+1}$ make the diagram \eqref{n-1 tuple of homs} commute.
Therefore they give rise to an element $\phi\in
\textrm{Hom}(\CF'(\infty), \CO^n/\CF)$. \\

It is easily seen that the procedures $\phi\rightarrow \tilde{\phi}$
and $\tilde{\phi}\rightarrow \phi$ described above are inverse to
each other, and thus they give the desired natural isomorphism
\eqref{natural iso}. Thus we may conclude that $E$ is the sheaf
whose fiber above $(\CF, \CF')$ is
\begin{equation} \label{reinterpret E}
E|_{(\CF,\CF')}=\textrm{Hom}_z(\tilde{\CF}'(D_{\infty}),\CO_{\BP^1\times
\BA^1}^n/\tilde{\CF}).
\end{equation} 

\subsection{} For any $i\geq 2$ and any coherent sheaves $\CS,\CS'$ on $\BP^1\times \BA^1$, we have $\textrm{Ext}^i(\CS, \CS')=0$.
This can be easily seen by covering $\BP^1\times \BA^1$ with the
affine charts $(\BP^1-0)\times \BA^1$ and $(\BP^1-\infty)\times
\BA^1$, and then using the \v Cech complex to compute $\textrm{Ext}^i(\CS, \CS')$. Then let us define the ``virtual vector space''
\begin{equation}
\chi(\CS, \CS')=\textrm{Hom}(\CS, \CS')-\textrm{Ext}^1(\CS, \CS')
\end{equation}
If $\CS$ and $\CS'$ are $z-$invariant sheaves on
$\BP^1\times \BA^1$, we can take the $z-$invariant parts of the
$\textrm{Ext}$ groups in question, and define
\begin{equation}
\chi_z(\CS, \CS')=\textrm{Hom}_z(\CS, \CS')-\textrm{Ext}^1_z(\CS,\CS')
\end{equation}
The advantage of the functor $\chi$ over the functors $\textrm{Ext}^i$ is the following additivity property: suppose $\CS, \CS', \CS'', \CS_0$ are coherent sheaves such that the sequence
$$
0\rightarrow \CS'\rightarrow \CS\rightarrow \CS''\rightarrow 0
$$
is exact. Then we have the following relations in the Grothendieck group:
$$
\chi(\CS_0, \CS)= \chi(\CS_0, \CS')+\chi(\CS_0, \CS'')
$$
\begin{equation} \label{additivity of chi}
\chi(\CS,\CS_0)= \chi(\CS', \CS_0)+\chi(\CS'', \CS_0)
\end{equation}
The same additivity property holds for
the functor $\chi_z$, when $\CS, \CS', \CS'', \CS_0$ are
$z-$invariant sheaves. Moreover, we have
\begin{equation} \label{twist of chi}
\chi(\CS\otimes L, \CS'\otimes L)=\chi(\CS, \CS')
\end{equation}
for any coherent sheaves $\CS,\CS'$ and line bundle $L$. The same
property holds for $\chi_z$, when $\CS, \CS'$ and $L$ are
$z-$invariant. \\

\subsection{} The above natural properties of the functors $\chi$ and $\chi_z$ will allow us to reinterpret the sheaf $E$.
To do this, we will need the following technical result:
\begin{lemma} \label{general result}
Let $\CS, \CS'$ be locally free sheaves on $\BP^1$. Then for every
$i\geq 0$ we have natural isomorphisms
\begin{equation} \label{computation chi z k<=0}
\emph{Ext}^i_z(\pi^*\CS'(kL_0), \pi^*\CS) \cong \emph{Ext}^i(\CS',
\CS), \emph{ if } k\leq 0
\end{equation}
\begin{equation} \label{computation chi z k>0}
\emph{Ext}^i_z(\pi^*\CS'(kL_0), \pi^*\CS) \cong 0, \emph{ if } k > 0
\end{equation}
\end{lemma}

\begin{proof} It is enough to prove the lemma for $i=0$, namely for the
functor $\textrm{Hom}$. This is because the same argument will prove the Lemma for arbitrary $i$,
by using the \v Cech complexes with respect to the affine coverings
$$
\BP^1\times \BA^1=(\BP^1-0)\times \BA^1 \cup (\BP^1-\infty)\times
\BA^1, \textrm{ }\textrm{ }\textrm{ }\textrm{ } \BP^1=(\BP^1-0) \cup
(\BP^1-\infty)
$$
to compute the functors $\textrm{Ext}^i_z$ and $\textrm{Ext}^i$,
respectively. \\

So let us describe a $z-$invariant homomorphism from
$\pi^*\CS'(kL_0)$ to $\pi^*\CS$. The local sections of
$\pi^*\CS'(kL_0)$ are generated over $z$ by expressions of the form
$P(x)z^{-k}$, where $x$ is a local coordinate on $\BP^1$ and $P(x)$
is a local section of $\CS'$. If we want to map this in a
$z-$invariant way, we must send it to some expression of the form
$Q(x)z^{-k}$. But if $k>0$, there simply are no such local sections
in $\pi^*\CS$, which proves \eqref{computation chi z k>0}. \\

If, on the other hand, $k\leq 0$, then any $z-$invariant morphism is
completely determined by sending local sections $P(x)z^{-k}$ of
$\pi^*\CS'(kL_0)$ to local sections $Q(x)z^{-k}$ of $\pi^*\CS$. This
amounts to sending local sections $P(x)$ of $\CS'$ to local sections
$Q(x)$ of $\CS$. But this is just the data that defines a
homomorphism from $\CS'$ to $\CS$, and this proves
\eqref{computation chi z k<=0}.\\
\end{proof}

This lemma has a very useful corollary. Consider the following
resolution of $\tilde{\CF}$:
\begin{equation} \label{locally free resolution}
0\rightarrow \bigoplus_{i=1}^{n-1} \pi^*\CF_i(-iL_0)\rightarrow
\bigoplus_{i=1}^n \pi^*\CF_i(-(i-1)L_0)\rightarrow
\tilde{\CF}\rightarrow 0
\end{equation}
where we write $\CF_n=\CO^n$ by convention. The second non-zero arrow sends the direct sum to the sum of
sheaves, while the first non-zero arrow sends the summand
$\pi^*\CF_i(-iL_0)$ into $\pi^*\CF_i(-(i-1)L_0)$ via the standard
embedding $\CO(-L_0)\hookrightarrow \CO$ and into
$\pi^*\CF_{i+1}(-iL_0)$ via minus the embedding
$\CF_i\hookrightarrow \CF_{i+1}$.

\begin{lemma} \label{intermediary 0} For any flag $\CF\in \CM$ and any $i\geq 0$, we have
$$
\emph{Ext}^i_z(\tilde{\CF}(D_\infty), \CO_{\BP^1\times \BA^1})=0
$$
\end{lemma}
\begin{proof} Obviously, it is enough to prove this result for
$i=0,1$. For $i=0$, in a similar way with \eqref{natural iso} one
shows that
$$
\textrm{Hom}_z(\tilde{\CF}(D_{\infty}), \CO_{\BP^1\times
\BA^1})\cong \textrm{Hom}(\CF(\infty), \CO)
$$
By definition, the space in the right hand side consists of commutative
diagrams of morphisms of sheaves on $\BP^1$:
$$
\begin{CD}
\CF_1(\infty) @>{\phi_1}>> \CO \\
@VVV @VVV \\
... @>>> ... \\
@VVV @VVV \\
\CF_{n-1}(\infty) @>{\phi_{n-1}}>> \CO \\
@VVV @VVV \\
\CO(\infty) @>{\phi_n}>> \CO \\
\end{CD}
$$
From the above commutative diagram, we see that all the maps
$\phi_1$,...,$\phi_{n-1}$ are determined by $\phi_n$ via restriction. But because of the torsion with $\infty$, we must have
$\phi_n=0$. Thus, there are no non-trivial elements of
$\textrm{Hom}(\CF(\infty), \CO)$, and this proves the lemma for
$i=0$. \\

Now apply relation \eqref{additivity of chi} to the short exact
sequence \eqref{locally free resolution} (twisted by $D_\infty$). We
obtain
$$
\chi_z(\tilde{\CF}(D_\infty), \CO_{\BP^1\times
\BA^1})= \sum_{k=1}^n
\chi_z(\pi^*\CF_k(D_\infty-(k-1)L_0), \CO_{\BP^1\times
\BA^1}) -
$$
\begin{equation} \label{relation 999}
- \sum_{k=1}^{n-1}
\chi_z(\pi^*\CF_k(D_\infty-kL_0), \CO_{\BP^1\times
\BA^1})
\end{equation}
To prove the Lemma for $i=1$, it is enough to show that
$\textrm{Ext}^1_z(\tilde{\CF}(D_\infty), \CO_{\BP^1\times
\BA^1})$ has dimension 0. For this, it
is enough to show that the virtual vector space
$\chi_z(\tilde{\CF}(D_\infty), \CO_{\BP^1\times \BA^1})$ has
dimension 0. Thus it is enough to show that the two sums
of vector spaces in \eqref{relation 999} have the same
dimension. But from Lemma~\ref{general result} we have
$$
\sum_{k=1}^n \chi_z(\pi^*\CF_k(D_\infty-(k-1)L_0),
\CO_{\BP^1\times \BA^1})=\sum _{k=1}^n \chi(\CF_k(\infty),
\CO)
$$
$$
\sum_{k=1}^{n-1} \chi_z(\pi^*\CF_k(D_\infty-kL_0),
\CO_{\BP^1\times \BA^1})= \sum_{k=1}^{n-1}
\chi(\CF_k(\infty), \CO)
$$
Since $\chi(\CO^n(\infty), \CO)=0$, the above sums are
equal, and thus they have the same dimension. This concludes the proof.
\end{proof}

\subsection{} Now to reinterpret $E$. The short exact sequence
$$
0\rightarrow \tilde{\CF}\rightarrow \CO_{\BP^1\times \BP^1}^n\rightarrow \CO_{\BP^1\times \BP^1}^n/\tilde{\CF}\rightarrow 0
$$
induces the long exact sequence of $\textrm{Ext}_z$ groups:
$$
\textrm{Hom}_z(\tilde{\CF}'(D_\infty), \CO_{\BP^1\times
\BP^1}^n)\rightarrow \textrm{Hom}_z(\tilde{\CF}'(D_\infty),
\CO_{\BP^1\times \BP^1}^n/\tilde{\CF})\rightarrow
$$
\begin{equation} \label{long exact sequence}
\rightarrow \textrm{Ext}^1_z(\tilde{\CF}'(D_\infty),
\tilde{\CF})\rightarrow \textrm{Ext}^1_z(\tilde{\CF}'(D_\infty),
\CO_{\BP^1\times \BP^1}^n)
\end{equation}
Lemma~\ref{intermediary 0} implies that the vector spaces at
the ends of the sequence are 0, and therefore
\begin{equation} \label{relation 998}
\textrm{Hom}_z(\tilde{\CF}'(D_\infty), \CO_{\BP^1\times \BP^1}^n/\tilde{\CF}) \cong \textrm{Ext}^1_z(\tilde{\CF}'(D_\infty), \tilde{\CF})
\end{equation}
However, we have that $\textrm{Hom}_z(\tilde{\CF}'(D_\infty), \tilde{\CF})=0$. This is proven word by word as the $i=0$ statement of
Lemma~\ref{intermediary 0}, so we will not repeat the proof here. Therefore, relations \eqref{relation 998} and \eqref{reinterpret E} imply that $E$
is the sheaf with fiber over $(\CF, \CF')$:
\begin{equation} \label{reinterpret E 3}
E|_{(\CF, \CF')}=-\chi_z(\tilde{\CF}'(D_\infty), \tilde{\CF}).
\end{equation} 
From Corollary 7.9.9. of \cite{GR}, we know that the dimension of the virtual vector spaces $\chi_z(\tilde{\CF}'(D_\infty), \tilde{\CF})$ is a locally constant function of $\CF$ and $\CF'$. Therefore, so is the dimension of the fibers of $E$. This implies that $E$ is a vector bundle, and its rank can be computed from the character of $T\times \BC^*$ in the fixed fibers of $E$ (just set $x,a_1,...,a_n \rightarrow 0$ in Proposition~\ref{explicit char computation}). Summarizing everything, we have: 
\begin{proposition} 
\label{E vector bundle}
The sheaf $E$ is a vector bundle on $\CM\times \CM$, of rank equal to half the dimension of the base $\CM\times \CM$.
\end{proposition}

\subsection{} We will use the resolutions \eqref{locally free resolution} for $\tilde{\CF}$ and $\tilde{\CF}'$ to compute the $K-$theoretic class of
$E$. Properties \eqref{additivity of chi} and \eqref{twist of chi}
of the functor $\chi$ imply the following relations in the
Grothendieck group of vector spaces:
$$
E|_{(\CF, \CF')}=-\chi_z(\tilde{\CF}'(D_\infty), \tilde{\CF})=
$$
$$
=-\sum_{i=1}^n\sum_{j=1}^n \chi_z(\pi^*\CF_i'(D_{\infty}+(j-i)L_0), \pi^*\CF_j)+\sum_{i=1}^n\sum_{j=1}^{n-1} \chi_z(\pi^*\CF_i'(D_{\infty}+(j-i+1)L_0), \pi^*\CF_j)+
$$
$$
+\sum_{i=1}^{n-1}\sum_{j=1}^n \chi_z(\pi^*\CF_i'(D_{\infty}+(j-i-1)L_0), \pi^*\CF_j)-\sum_{i=1}^{n-1}\sum_{j=1}^{n-1} \chi_z(\pi^*\CF_i'(D_{\infty}+(j-i)L_0), \pi^*\CF_j)
$$
Lemma~\ref{general result} allows us to compute the above spaces, yielding
$$
E|_{(\CF, \CF')}=-\sum_{j\leq i\leq n} \chi(\CF_i'(\infty), \CF_j) +\sum_{j\leq i-1\leq n-1} \chi(\CF_i'(\infty), \CF_j)+\sum_{j\leq i+1\leq n} \chi(\CF_i'(\infty), \CF_j)-
$$
\begin{equation} \label{reinterpret E 4}
-\sum_{j\leq i\leq n-1} \chi(\CF_i'(\infty), \CF_j)=\sum_{i=1}^{n-1} \chi(\CF_i'(\infty), \CF_{i+1})- \sum_{i=1}^{n-1}\chi(\CF_i'(\infty), \CF_i)
\end{equation}
Since all the identifications we have made along the way are natural, the above
equalities paste over all $(\CF,\CF')$ to give the following equality
in the Grothendieck group of coherent sheaves on $\CM\times \CM$:
\begin{equation} \label{reinterpret E 5}
E=\sum_{i=1}^{n-1} G_i-\sum_{i=1}^{n-1} H_i
\end{equation}
In the above, $G_i$ and $H_i$ are the sheaves whose fibers over $(\CF, \CF')$ are $\chi(\CF_i'(\infty), \CF_{i+1})$ and $\chi(\CF_i'(\infty), \CF_i)$, respectively. 
By Corollary 7.9.9. of \cite{GR}, they are vector bundles. \\

\subsection{} Relation \eqref{reinterpret E 5} allows us to compute the character of $T\times \BC^*$ in the fixed fibers of $E$. Knowing the character implies knowing the weights of the $T\times \BC^*$ action, since the character is just the sum of the exponentials of the weights. But knowing the weights implies knowing the equivariant Chern classes of $E$, by Theorem~\ref{localization formula}. Therefore, let $\CF,\CF'$ be fixed points of $\CM$ corresponding to the vectors of non-negative integers $d,d'$ as in \eqref{fixed flag vector of integers}, and let us denote the fiber of $E$ above $(\CF,\CF')$ by $E_{d,d'}$.

\begin{proposition} \label{explicit char computation}
With the above notations, the character $\emph{char}(E_{d,d'})$ of $T\times \BC^*$ in the fiber $E_{d,d'}$ equals
$$
\sum_{i=1}^{n-1}\sum_{j=1}^{i+1}\sum_{j'=1}^i \frac {e^{a_j}}{e^{a_{j'}}}\cdot \frac {e^{x({d'}_i^{j'}-d_{i+1}^j+1)}-e^x}{e^x-1} -\sum_{i=1}^{n-1}\sum_{j=1}^{i}\sum_{j'=1}^i \frac {e^{a_j}}{e^{a_{j'}}}\cdot \frac {e^{x({d'}_i^{j'}-d_{i}^j+1)}-e^x}{e^x-1}
$$
\end{proposition}

\begin{proof} If we have a short exact sequence of representations of $T\times \BC^*$:
$$
0\rightarrow U\rightarrow V\rightarrow W\rightarrow 0
$$
then it is easy to see that the character is additive:
$$
\textrm{char}(V)=\textrm{char}(U)+\textrm{char}(W)
$$
Therefore, by \eqref{reinterpret E 5}, it is enough to show that
\begin{equation} \label{interrrrm}
\textrm{char}(G_i|_{d,d'})=\sum_{j=1}^{i+1}\sum_{j'=1}^i \frac {e^{a_j}}{e^{a_{j'}}}\cdot \frac {e^{x({d'}_i^{j'}-d_{i+1}^j+1)}-e^x}{e^x-1}
\end{equation}
\begin{equation}
\textrm{char}(H_i|_{d,d'})=\sum_{j=1}^{i}\sum_{j'=1}^i \frac {e^{a_j}}{e^{a_{j'}}}\cdot \frac {e^{x({d'}_i^{j'}-d_{i}^j+1)}-e^x}{e^x-1}
\end{equation}
for each $1\leq i\leq n-1$. We will prove the first of these statements, as the second one is completely analogous. By \eqref{fixed flag}, we have
$$
\CF_i'=w_1\cdot \CO(-{d'}_i^1)\oplus... \oplus w_i\cdot \CO(-{d'}_i^i)
$$
$$
\CF_{i+1}=w_1\cdot \CO(-{d}_{i+1}^1)\oplus... \oplus w_{i+1}\cdot \CO(-{d}_{i+1}^{i+1})
$$
By \eqref{additivity of chi}, we therefore have
$$
G_i|_{d,d'}=\chi(\CF_i'(\infty), \CF_{i+1})=\sum_{j=1}^{i+1}\sum_{j'=1}^i \chi(w_{j'}\cdot \CO(\infty-{d'}_i^{j'}), w_j\cdot \CO(-d_{i+1}^j))
$$
The character of $\BC^*$ acting in $\chi(\CO(\infty+k), \CO(l))$ is easily seen to be $\frac {e^{x(l-k+1)}-e^x}{e^x-1}$. Therefore, the character of $T\times \BC^*$ in $\chi(w_{j'}\cdot \CO(\infty-{d'}_i^{j'}), w_j\cdot \CO(-d_{i+1}^j))$ equals
$$
\frac {e^{a_j}}{e^{a_{j'}}}\cdot \frac {e^{x({d'}_i^{j'}-d_{i+1}^j+1)}-e^x}{e^x-1}
$$
Summing this up over all $j,j'$ implies the desired relation \eqref{interrrrm}.
\end{proof}
\textrm{ }\\

\section{The Operators $A(m)$}
\label{the operators A(m)}

\subsection{} Given a $T\times \BC^*-$vector bundle $V$ of rank $r$, its \emph{Chern polynomial} is defined as
$$
c(V, t)=c_r(V)+c_{r-1}(V)\cdot t+c_{r-2}(V)\cdot t^2+...+c_0(V)\cdot t^{r}
$$
where $c_0(V)=1, c_1(V), ..., c_r(V)$ are the Chern classes of $V$. The top Chern class will always be denoted by $c_r(V)=e(V)$, as in Section~\ref{equiv cohomology}. In case $V$ is a $T\times \BC^*-$vector bundle over a point (i.e. a representation of $T\times \BC^*$), then we have
$$
c(V,t)=\prod_{w\in V} (w+t)
$$
In the above, $\prod_{w\in V}$ denotes the product over the weights of $T\times \BC^*$ acting in $V$. As in Section~\ref{properties of equivariant cohomology}, these weights are all elements of $H^*_{T\times \BC^*}(pt)$.\\

By Proposition~\ref{E vector bundle}, the sheaf $E$ introduced in the previous section is a vector bundle. Take a complex number $m$ and let $x\in H^*_{T\times \BC^*}(pt)$ be as in Section~\ref{equiv cohom quasiflags}. Then we define the operator $A(m):H\rightarrow H$ by

\begin{equation}
A(m)(\alpha)=p_*(c(E, mx)\cdot q^*\alpha)
\end{equation}
where  $p,q:\CM\times \CM \rightarrow \CM$ are the standard projections
onto the first and second factors, respectively. \\

\subsection{} By Theorem~\ref{localization formula}, we have
$$
c(E, mx)\cdot q^*[d]=\sum_{d'} [d',d]\cdot \frac {c(E,
mx)|_{d',d}}{\dsp \prod_{w\in \CT_{d'} \CM} w}=\sum_{d'} [d',d]\cdot \frac
{\dsp \prod_{w\in E_{d',d}} (w+mx)}{\dsp \prod_{w\in \CT_{d'} \CM} w}
$$
Clearly, $p_*[d',d]=[d']$ and thus
\begin{equation} \label{matrix coefficients of A(m) formula}
A(m)[d]=\sum_{d'} [d']\cdot \frac
{\dsp \prod_{w\in E_{d',d}} (w+mx)}{\dsp \prod_{w\in \CT_{d'} \CM} w}
\end{equation}
Recall that the generalized character of $A(m)$ is, by the definition in Section~\ref{def generalized character},
\begin{equation} \label{char of A(m)}
\chi_{A(m)}=\sum_{\lambda } e^{\lambda}\cdot
\textrm{Tr}(A(m)|_{H[\lambda]})
\end{equation}
We will evaluate the above traces in the basis $[d]$ of $H$. Recall from \eqref{weights} that a basis vector $[d]$ lies in the $\lambda-$weight space of $H$ if and only if
$$
\lambda=\frac ax-\rho+\gamma, \textrm{ }\textrm{ where }\textrm{ } \gamma=-d_1\alpha_1-...-d_{n-1}\alpha_{n-1}\in Q^-
$$
Then \eqref{matrix coefficients of A(m) formula} implies that
$$
\chi_{A(m)}=\sum_{\gamma\in Q^-} e^{\frac ax-\rho+\gamma}\cdot \sum_{d\in \CM_{\gamma}^{T\times \BC^*}} \frac
{\dsp \prod_{w\in E_{d,d}} (w+mx)}{\dsp \prod_{w\in \CT_{d} \CM} w}=
$$
By using the isomorphism $E_{d,d}\cong \CT_{d} \CM$ of Proposition~\ref{E tangent bundle}, the above becomes
\begin{equation}\label{chi A(m)}
\chi_{A(m)}=\sum_{\gamma\in Q^-} e^{\frac ax-\rho+\gamma}\cdot \sum_{d\in \CM_{\gamma}^{T\times \BC^*}} \prod_{w\in \CT_{d} \CM}  \frac
{w+mx}{w}
\end{equation}
But Corollary~\ref{integral formula} gives us
\begin{equation} \label{coeffs Z(m)}
\int_{\CM_{\gamma}} c(\CT\CM_{\gamma},mx)=\sum_{d\in
\CM_{\gamma}^{T\times \BC^*}} \prod_{w\in \CT_{d} \CM}  \frac {w+mx}{w}
\end{equation}
Therefore relations \eqref{chi A(m)} and \eqref{coeffs Z(m)} imply the following:
\begin{proposition} \label{character and gen function}
The character of $A(m)$ is related to the generating function $Z(m)$ of \eqref{generating function} by
$$
\chi_{A(m)}=\sum_{\gamma\in Q^-} e^{\frac ax-\rho+\gamma}\cdot \int_{\CM_{\gamma}}
c(\CT\CM_{\gamma},mx)=e^{\frac ax-\rho} \cdot Z(m).
$$
\end{proposition}
\textrm{ } \\

\subsection{} We will now seek to relate the operator $A(m)$ to the $\fsl_n-$action on
$H$ given by the operators $e_i,f_i$ of Section~\ref{action on H}.

\begin{proposition} \label{factorization A(0)}
The operator $A(0)$ belongs to the $SL_n-$action, and is precisely
$$
A(0)=\prod_{\theta \in R^+} \exp(e_{\theta})=\left(
                                                       \begin{array}{ccccc}
                                                         1 & 1 & 1 & ... & 1  \\
                                                         0 & 1 & 1 & ... & 1  \\
                                                         0 & 0 & 1 & ... & 1  \\
                                                         ... & ... & ... & ... &...  \\
                                                         0 & 0 & 0 & ... & 1  \\
                                                       \end{array}
                                                     \right)=:g \in SL_n
$$
The order of the terms in the product corresponds to the descending order of the positive roots
$$
w_{n-1}-w_n > w_{n-2}-w_n > w_{n-2}-w_{n-1} > ... > w_1-w_n > w_1-w_{n-1} > ... > w_1-w_2
$$
\end{proposition}

\begin{remark} A priori, the $\fsl_n-$action on $H$ doesn't immediately give a well-defined $SL_n$ action on $H$, but only on a suitable completion of $H$. However, since $H$ is isomorphic to the Verma module, any vector in $H$ is annihilated by a high enough power of the $e_i$'s. Therefore the ``upper triangular'' part of $SL_n$ does act correctly on the space $H$.
\end{remark}

\begin{proof}
The second equality is easily proved as an equality in $SL_n$.
Indeed, $e_{w_i-w_j}$ corresponds to
the matrix $E_{i,j}$ which has entry 1 at the intersection of row
$i$ and column $j$, and 0 everywhere else. Then we have
$$
\prod_{\theta \in R^+} \exp(e_{\theta})=\prod_{i=n-1}^{1}\prod_{j=n}^{i+1} \exp(E_{i,j})=\prod_{i=n-1}^{1}\prod_{j=n}^{i+1} (1+E_{i,j})
$$
But one easily notes that if $i>i'$ or if $i=i'$ and $j>j'$, then $E_{i,j}\cdot E_{i',j'}=0$. This means that all products of more than one $E_{i,j}$ in the above expression vanish, and therefore
$$
\prod_{\theta \in R^+} \exp(e_{\theta})=1+\sum_{i=n-1}^{1}\sum_{j=n}^{i+1} E_{i,j}=\left(
                                                       \begin{array}{ccccc}
                                                         1 & 1 & 1 & ... & 1  \\
                                                         0 & 1 & 1 & ... & 1  \\
                                                         0 & 0 & 1 & ... & 1  \\
                                                         ... & ... & ... & ... &...  \\
                                                         0 & 0 & 0 & ... & 1  \\
                                                       \end{array}
                                                     \right)
$$
Let us now turn to the proof of the first equality in the statement of Proposition~\ref{factorization A(0)}.

\begin{lemma} \label{transversal section}
There exists a regular section of $E$ which vanishes on the locus
$$
\fC=\{(\CF,\CF')|\CF'\subset \CF\}
$$
\end{lemma}
\begin{proof} Consider the composition
$$
(\CF'\hookrightarrow \CO^n\rightarrow \CO^n/\CF) \in \textrm{Hom}(\CF',\CO^n/\CF)
$$
where the first map is the inclusion of $\CF'$ in $\CO^n$, and the second is projection onto the quotient. This morphism vanishes above $\infty\in \BP^1$, so it naturally extends to a morphism $r(\CF, \CF')\in \textrm{Hom}(\CF'(\infty),\CO^n/\CF)$. Thus we have obtained a section $r$ of the bundle $E$, which is easily seen to vanish precisely when $\CF'\subset \CF$. Therefore, the zero locus of the section $r$ is $\fC$. Since the rank of $E$ is half the dimension of the base $\CM\times \CM$ (by Proposition~\ref{E vector bundle}), then the dimension of each irreducible component of the zero locus $\fC$ will be at least half the dimension of $\CM \times \CM$. Together with Proposition~\ref{irreducible components}, this implies:

\begin{proposition} \label{irreducible components 2}
The distinct irreducible components of $\fC$ are precisely the subvarieties $\fC_{\kappa}$, and all have dimension equal to half the dimension of $\CM\times \CM$.
\end{proposition}

Now to prove that the section is regular, we must prove that $r$ is transversal at the generic point of the zero locus. In other words, the image of the section $r$ and the image of the zero section must be transversal in the total space $\CE$ of $E$. Given a point $(\CF,\CF')\in \fC_\kappa$ with $\CF'\subset \CF$, we have the following relation of tangent spaces
$$
\CT_{(\CF,\CF',0)} \CE=\CT^h\oplus \CT^v
$$
In the above decomposition, $\CT^h$ (the horizontal tangent space) is just the pull-back of $\CT_{(\CF,\CF')}\CM \times \CM$, while $\CT^v$ (the vertical tangent space) is isomorphic to the bundle $E$ itself. The horizontal tangent space is spanned by vectors tangent to the image of the zero section. It is therefore enough to show that the vertical tangent space is spanned by vectors tangent to the image of $r$, for the generic point $(\CF,\CF')\in \fC_\kappa$. Let $p_v:\CT^h\oplus \CT^v\rightarrow \CT^v\cong E|_{(\CF, \CF')}$ be the projection onto the second factor. Then what we must show is that the map $p_v\circ r_*$ maps $T_{(\CF,\CF')} \CM\times \CM$ surjectively onto $\CT^v$, and this will be equivalent to transversality at $(\CF, \CF')$. \\

By Proposition~\ref{E tangent bundle}, a tangent vector to $\CM$ at $\CF'$ is a homomorphism $\psi\in \textrm{Hom}(\CF'(\infty), \CO^n/\CF')$. Such a vector is also tangent to $\CM\times \CM$ at $(\CF,\CF')$; more precisely, it is tangent in the direction of the second factor of $\CM\times \CM$. The map $p_v\circ r_*$ maps $\psi$ to the homomorphism
$$
\pi\circ \psi=\CF'(\infty)\rightarrow \CO^n/\CF'\rightarrow \CO^n/\CF \in \textrm{Hom}(\CF'(\infty), \CO^n/\CF)
$$
where $\pi:\CO^n/\CF'\rightarrow \CO^n/\CF$ is the standard projection induced by the inclusion $\CF'\hookrightarrow \CF$. Then it is enough to show that the natural map
$$
\textrm{Hom}(\CF'(\infty), \CO^n/\CF')@>{\psi\rightarrow \pi\circ \psi}>> \textrm{Hom}(\CF'(\infty), \CO^n/\CF)
$$
is surjective for the generic point $(\CF,\CF')\in \fC_\kappa$. Take an element $\phi\in \textrm{Hom}(\CF'(\infty), \CO^n/\CF)$, which corresponds to a commutative diagram of homomorphisms $(\phi_1,...,\phi_{n-1})$ as in \eqref{n-1 tuple of homs}. Let $\pi_i:\CO^n/\CF'_i\rightarrow \CO^n/\CF_i$ be the standard projections. We have the exact sequence
$$
\textrm{Hom}(\CF'_i(\infty), \CO^n/\CF_i')\rightarrow \textrm{Hom}(\CF'_i(\infty), \CO^n/\CF_i)\rightarrow \textrm{Ext}^1(\CF'_i(\infty), \CF_i/\CF_i')=0
$$
The $\textrm{Ext}$ space on the right vanishes because $\CF'_i(\infty)$ is a torsion free sheaf on $\BP^1$, whereas $\CF_i/\CF'_i$ is a direct sum of skyscraper sheaves. Therefore, each homomorphism $\phi_i\in \textrm{Hom}(\CF'_i(\infty), \CO^n/\CF_i)$ can be extended to an element $\psi_i\in \textrm{Hom}(\CF'_i(\infty), \CO^n/\CF_i')$ such that $\phi_i=\pi_i\circ \psi_i$. In this way, we can extend the homomorphism $\phi$ to a homomorphism $\psi\in \textrm{Hom}(\CF'(\infty), \CO^n/\CF')$ such that $\phi=\pi\circ \psi$. The only problem is that, a priori, $\psi=(\psi_1,...,\psi_{n-1})$ does not make the diagram

$$
\begin{CD}
... @ .... \\
@VVV @VVV \\
\CF'_i(\infty) @>{\psi_i}>> \CO^n/\CF'_i \\
@V{\rho_i}VV @V{\nu_i}VV \\
\CF'_{i+1}(\infty) @>{\psi_{i+1}}>> \CO^n/\CF'_{i+1} \\
@VVV @VVV \\
... @ .... \\
\end{CD}
$$
commute (in the above, $\rho_i$ is the inclusion and $\nu_i$ is induced by the inclusion). In other words, we would like to have $\psi_{i+1}\circ \rho_i-\nu_i\circ \psi_i=0$. However, all that we can say a priori is that
$$
\pi_{i+1}\circ (\psi_{i+1}\circ \rho_i-\nu_i\circ \psi_i)=\phi_{i+1}\circ \rho_i-\nu_i\circ \phi_i=0,
$$
because the commutative diagram \eqref{n-1 tuple of homs} is known to commute. Therefore, $\psi_{i+1}\circ \rho_i-\nu_i\circ \psi_i$ is a homomorphism between $\CF'_i(\infty)$ and $\CF_{i+1}/\CF'_{i+1}$, and we want to perturb the $\psi_i$'s in such a way as to make this homomorphism 0. We will perform this perturbation by inductively adding to each $\psi_{i+1}$ a homomorphism $d\psi_{i+1}:\CF'_{i+1}(\infty) \rightarrow \CF_{i+1}/\CF'_{i+1}$ such that
$$
(\psi_{i+1}+d\psi_{i+1})\circ \rho_i-\nu_i\circ \psi_i=0.
$$
It is clear that this perturbation preserves the property $\pi\circ \psi=\phi$. To do this, it is enough to show that every homomorphism between $\CF'_i(\infty)$ and $\CF_{i+1}/\CF'_{i+1}$ can be written as $d\psi_{i+1}\circ \rho_i$. In other words, it is enough to prove the surjectivity of the natural restriction map
$$
\textrm{Hom}(\CF'_{i+1}(\infty), \CF_{i+1}/\CF'_{i+1})\rightarrow \textrm{Hom}(\CF'_i(\infty), \CF_{i+1}/\CF'_{i+1})
$$
But the cokernel of this map is contained in
$$
\textrm{Ext}^1(\CF'_{i+1}/\CF'_i, \CF_{i+1}/\CF'_{i+1})
$$
and therefore it is enough to show that this \textrm{Ext} group is 0. Recall that $\CF_{i+1}/\CF'_{i+1}$ is a torsion sheaf, supported at finitely many points. Moreover, $\CF'_{i+1}/\CF'_i$ has finitely many ``torsion points'', i.e. points of $\BA^1=\BP^1 \backslash \infty$ above which there is torsion. In order for the above $\textrm{Ext}^1$ group to vanish, it would be enough to have the support of $\CF_{i+1}/\CF'_{i+1}$ and the torsion points of $\CF'_{i+1}/\CF'_i$ be two disjoint sets (just apply Serre duality). Therefore, Lemma~\ref{transversal section} reduces to the following:
\begin{lemma} For the generic point $(\CF,\CF')\in \fC_\kappa$, the support of the sheaf $\CF_{i+1}/\CF'_{i+1}$ is disjoint from the set of torsion points of $\CF'_{i+1}/\CF'_i$.
\end{lemma}
\begin{proof} Note that the condition in the statement of the lemma is an open condition. Therefore, since $\fC_\kappa$ is irreducible, it is enough to find at least one point with the desired property. We will prove this lemma for the case when $\kappa$ is an indecomposable quiver representation, since the case when $\kappa$ is a sum of such representations is completely analogous. \\

Suppose that $\kappa$ is the indecomposable quiver representation $j\rightarrow j+1\rightarrow...\rightarrow k\rightarrow 0$, and choose any flag $\CF\in p(\fC_\kappa)$. The sections of the flag components $\CF_i\subset \CO^n$ are (locally) vectors $(P_1,...,P_n)$ of polynomials in $x$. For each $i$, the inclusion $\CF_{i-1}\hookrightarrow \CF_i$ will be an inclusion of vector bundles in the vicinity of the generic point $\zeta\in \BA^1$. Near such a point $\zeta$, the bundle $\CF_{j-1}$ will be cut out in $\CF_j$ by a linear equation $Q_1P_1+...+Q_nP_n=0$. We can choose $\zeta$ such that there exist $(P_1,...,P_n)\in \CF_j$ which verify $x-\zeta\nmid Q_1P_1+...+Q_nP_n$. Then define $\CF'_i=\CF_i$ for $i<j$ or $i>k$, whereas for $j\leq i\leq k$ define
$$
\CF'_i=\{(P_1,...,P_n)\in \CF_i, \textrm{ such that } x-\zeta|  Q_1P_1+...+Q_nP_n\}
$$
Then one sees that $\CF'_{i-1}\subset \CF'_i$ for all $i$. Indeed, the only non-trivial inclusion is the one for $i=j$, but this one follows from the fact that $Q_1P_1+...+Q_nP_n=0$ for $(P_1,...,P_n)\in \CF_{j-1}$. The maps
$$
\CF_j/\CF'_j\rightarrow...\rightarrow \CF_k/\CF'_k\rightarrow 0
$$
induce the quiver representation $j\rightarrow j+1\rightarrow...\rightarrow k\rightarrow 0$. To see this, note that for $j\leq i\leq k-1$, we have that $\CF_i \supset \CF_j$ and $\CF_j$ contains an element $(P_1,...,P_n)$ such that $x-\zeta\nmid Q_1P_1+...+Q_nP_n$. Therefore $(P_1,...,P_n)\notin \CF'_{i+1}$, and thus $\CF_i$ is not a subsheaf of $\CF'_{i+1}$. Thus we have that
$$
(\CF,\CF')\in \fC_\kappa^\circ\subset \fC_\kappa
$$
Now to show that $(\CF,\CF')$ constructed in this way satisfies the conclusion of the lemma. Obviously, the sheaves $\CF_{i+1}/\CF'_{i+1}$ have non-empty support only for $j-1\leq i\leq k-1$, and in that case the support consists only of $\zeta$. On the other hand, we will show that for $j-1\leq i\leq k-1$ the sheaf $\CF'_{i+1}/\CF'_i$ does not have $\zeta$ as a torsion point. Indeed, take $(P_1,...,P_n)\in \CF'_{i+1}$ such that $((x-\zeta)P_1,...,(x-\zeta)P_n)\in \CF'_i$. Since $\zeta$ is not a torsion point of $\CF_{i+1}/\CF_i$ (by assumption on $\zeta$), then $(P_1,...,P_n)\in \CF_i$. However, since $(P_1,...,P_n)\in \CF'_{i+1}$ and $i+1\leq k$, this means that $x-\zeta|Q_1P_1+...+Q_nP_n$. Therefore, $(P_1,...,P_n)\in \CF'_i$ and thus the sheaf $\CF'_{i+1}/\CF'_i$ has no torsion at $\zeta$.
\end{proof}

This completes the proof of Lemma~\ref{transversal section}.

\end{proof}
\textrm{ }\\

Now let us return to the proof of Proposition~\ref{factorization A(0)}. Proposition~\ref{class of the section} and Lemma~\ref{transversal section} imply that
$$
c(E,0)=e(E)=\sum_{\kappa \in \textrm{ Rep}(A_{n-1})}[\fC_{\kappa}]
$$
Therefore, in the notation of Section~\ref{general correspondences}, we have
$$
A(0)=\sum_{\kappa \in \textrm{ Rep}(A_{n-1})} p_*([\fC_{\kappa}]\cdot q^*)=\sum_{\kappa \in \textrm{ Rep}(A_{n-1})} e_{\kappa} =
\sum_{b_1,...,b_{\nu}\geq 0} e_{b_1\theta_1+...+b_{\nu}\theta_{\nu}}
$$
where $\theta_{\nu}>...>\theta_1$ denote the positive roots. Applying Theorem~\ref{operators Ringel-Hall}, we know that the $e_{\kappa}$ satisfy the same multiplication rules as the corresponding elements of $U(A_{n-1})^{op}$. This and \eqref{formula quiver} imply that
$$
A(0)= \sum_{b_1,...,b_{\nu}\geq 0} \frac
{e_{\theta_{\nu}}^{b_{\nu}}}{b_{\nu}!}\cdot..\cdot \frac
{e_{\theta_{1}}^{b_{1}}}{b_{1}!}=\prod_{\theta\in R^+}
\exp(e_{\theta}).
$$\\
\end{proof}

\subsection{}
\label{enter s}

When $m\in \BN$, there is a direct way to factor $A(m)$ in terms of
$g:=A(0)$. For that, let $s:\CM \rightarrow \CM$ be the map
$$
s(\CF)=\CF(-m)
$$
In the above, $\CF(-m)$ denotes the flag obtained from $\CF$ by twisting each of
its components by $\CO(-m)=\CO(-m\cdot 0_{\BP^1})$. Then we have

\begin{proposition} \label{factorization A(m)}
For generic $\frac ax$,
$$
A(m)=y \cdot g \cdot s_*
$$
for some non-zero constant $y\in H^*_{T\times \BC^*}(pt)$.
\end{proposition}

\begin{remark} The word generic is mentioned because the constant $y$ has poles as a rational function of $\frac ax$.
\end{remark}

\begin{proof} The easiest way to prove this proposition is to use equivariant localization. By Theorem~\ref{locaization thm}, it is enough to prove that the desired equality holds for basis vectors $[d]\in H$, where $d=(d_j^i)$. Thus, we must show that
$$
A(m)[d]=y\cdot A(0)\cdot s_*[d]
$$
Obviously, $s_*[d]=[s(d)]=[d+m]$, where $d+m:=(d_j^i+m)$. By using \eqref{matrix coefficients of A(m) formula}, the above becomes equivalent to
$$
A(m)[d]=\sum_{d'} [d']\cdot \frac {\dsp \prod_{w\in E_{d',d}} (w+mx)}{\dsp \prod_{w\in \CT_{d'} \CM} w}=y\cdot \sum_{d'} [d']\cdot \frac {\dsp \prod_{w\in E_{d',d+m}} w}{\dsp \prod_{w\in \CT_{d'} \CM} w}=y\cdot A(0)\cdot s_*[d]
$$
Therefore, we must show that for any fixed points $d,d'$ we have
$$
\prod_{w\in E_{d',d}} (w+mx)=y\cdot \dsp \prod_{w\in E_{d',d+m}} w
$$
for some non-zero constant $y\in H^*_{T\times \BC^*}(pt)$ which is independent of $d,d'$. Since in any representation of $T\times \BC^*$, the character is the sum of the exponentials of the weights, the above relation is equivalent to
$$
\textrm{char}(E_{d',d}) \cdot e^{mx} = \textrm{constant} + \textrm{char}(E_{d',d+m})
$$
where the constant does not depend on $d,d'$. This immediately follows from Proposition~\ref{explicit char computation}.\\
\end{proof}

\subsection{} The map $s_*$ has a number of nice commutation properties with the generators of the $\fsl_n-$action on $H$.

\begin{proposition} \label{commutation s}
The map $s_*$ commutes with $f_1,...,f_{n-1}, e_1,..., e_{n-2}$. Moreover, its commutator with $h_{n-1}$ is given by
\begin{equation} \label{commutation s, h_n-1}
[s_*, h_{n-1}]=mn\cdot s_*
\end{equation}
\end{proposition}
\begin{proof} We will prove that $s_*$ commutes with $e_i$, for $1\leq i\leq n-2$. The proof of commutativity with the $f_i$'s will be completely analogous. Recall that $e_i=\tilde{p}_*\tilde{q}^*$, where $\tilde{p},\tilde{q}$ are the standard projections from $\fC_i\subset \CM\times \CM$ onto the two factors. Consider the correspondence
$$
\fC_i'=\{(\CF, \CF')\in \CM \times \CM|\CF'(-m) \subset_i \CF\}
$$
and let $\widehat{p}, \widehat{q}:\fC_i'\rightarrow \CM$ be the standard projections. Moreover, let $s': \fC_i'\rightarrow \fC_i$ be the map given by $s'(\CF, \CF')=(\CF, \CF'(-m))$. \\

From the definitions, it is easily seen that
$$
\begin{CD}
\CM @<{\widehat{p}}<< \fC_i' @>{\widehat{q}'}>> \CM \\
@ V{\textrm{Id}}VV @V{s'}VV @V{s}VV \\
\CM @<{\tilde{p}}<< \fC_i @>{\tilde{q}}>> \CM
\end{CD}
$$
commutes, and moreover the square on the right is a fiber square. Therefore we have the base-change formula $\tilde{q}^*s_*=s'_*{\widehat{q}}^*$, and thus
\begin{equation} \label{formula 1}
e_i\cdot s_*=\tilde{p}_*\tilde{q}^*s_*=\tilde{p}_*s'_*{\widehat{q}}^*=\widehat{p}_*{\widehat{q}}^*
\end{equation}
Now let $s'': \fC_i\rightarrow \fC_i'$ be the map given by $s''(\CF, \CF')=(\CF(-m), \CF')$. It gives us the commutative diagram
$$
\begin{CD}
\CM @<{\tilde{p}}<< \fC_i @>{\tilde{q}}>> \CM \\
@ V{s}VV @V{s''}VV @V{\textrm{Id}}VV \\
\CM @<{\widehat{p}}<< \fC_i' @>{\widehat{q}}>> \CM
\end{CD}
$$
There is no fiber square here, but note that $s''$ is an isomorphism, because the inverse map $(\CF, \CF')\rightarrow (\CF(m), \CF')$ is well-defined. Note that this only holds for $1\leq i\leq n-2$, since otherwise we cannot guarantee the fact that $\CF(m)\subset \CO^n$.

Since $s''$ is an isomorphism, therefore $s''_*{s''}^*=\textrm{Id}$. The above commutative diagram then gives us
\begin{equation} \label{formula 2}
s_*\cdot e_i=s_*\tilde{p}_*\tilde{q}^*=\widehat{p}_*{s''}_*\tilde{q}^*=\widehat{p}_*{s''}_*{s''}^*{\widehat{q}}^*=\widehat{p}_*{\widehat{q}}^*
\end{equation}
Thus \eqref{formula 1} and \eqref{formula 2} imply that $e_i\cdot s_*=\widehat{p}_*{\widehat{q}}^*=s_*\cdot e_i$. \\

To prove \eqref{commutation s, h_n-1}, we will check it on any basis vector $[d]$, where $d=(d_j^i)$ is a vector of non-negative integers. Since $s$ takes the fixed point $d$ to the fixed point $d+m=(d_j^i+m)$, we will have $s_*[d]=[d+m]$. The explicit formula for $h_{n-1}$ in \eqref{action of h} gives us
$$
[s_*,h_{n-1}]([d])=\left(\frac {a_{n-1}}x-\frac {a_n}x-1+d_{n-2}-2d_{n-1}\right)\cdot [d+m] -
$$
$$
- \left(\frac {a_{n-1}}x-\frac {a_n}x-1+d_{n-2}-2d_{n-1}-mn\right)\cdot [d+m]=mn\cdot s_*[d]
$$

\end{proof}

The properties of $s_*$ listed in the previous proposition are very important, because they completely determine the map $s_*$ (up to a constant).
\begin{proposition} \label{s_* unique}
All maps $\Psi:H\rightarrow H$ which commute with the action of $f_1,...,f_{n-1},e_1,...,e_{n-2}$ and satisfy
$$
[\Psi, h_{n-1}]=mn\cdot \Psi
$$
are constant multiples of $s_*$.
\end{proposition}
\begin{proof} Let us first compute $\Psi([0])$, where $0=(0,...,0)$. We will write $\Psi([0])$ in the basis of cohomology classes $[d]$:
\begin{equation} \label{psi 0}
\Psi([0])=\sum_d \alpha_d\cdot [d]
\end{equation}
The operators $h_1,...,h_{n-1}$ are diagonal in the basis $[d]$. Recall that $\Psi$ commutes with $h_1,...,h_{n-2}$ and $[\Psi, h_{n-1}]=mn\cdot \Psi$. Therefore, \eqref{action of h} implies that the $d$'s that appear with non-zero coefficient in the above sum have degrees $d_1=m$, $d_2=2m$, ..., $d_{n-1}=(n-1)m$. \\

Now suppose that there is a $\tilde{d}$ such that $\alpha_{\tilde{d}}\neq 0$ and $\tilde{d}_{j+1}^i<\tilde{d}_{j}^i$ for some $j\leq n-2$. If there are several such $\tilde{d}$'s, choose the one with the smallest $j$, and if there are several such $\tilde{d}$'s for the same $j$, then choose the one such that the number $\tilde{d}_j^i$ is minimal among all $i$. Since $\Psi$ commutes with $e_j$, it follows that
$$
e_j\left(\sum_d \alpha_d\cdot [d] \right)=0
$$
The coefficient of $[\tilde{d}-\delta_j^i]$ in the sum from the left hand side can only come from $e_j[\tilde{d}]$, by minimality of $\tilde{d}_j^i$. Since $\tilde{d}_{j+1}^i<\tilde{d}_{j}^i$, this coefficient is non-zero (see for example the computation of matrix coefficients of $e_j$ in ~\cite{BF}). Therefore, we obtain a contradiction with the assumption that $\alpha_{\tilde{d}}\neq 0$. \\

This means that in \eqref{psi 0}, we can have $\alpha_d\neq 0$ only if $d_j^i=d_{j+1}^i$ for all $1\leq i\leq j\leq n-2$. Since $d_1=m$, $d_2=2m$, ..., $d_{n-1}=m(n-1)$, the only $d$ that can appear with non-zero coefficient in ~\eqref{psi 0} is the one with $d_j^i=m, \forall i,j$. Therefore,
$$
\Psi([0])=\alpha\cdot[m]=\alpha\cdot s_*[0]
$$
for some $\alpha\in \BC$. But recall that $H$ is isomorphic to the Verma module, and thus the $f_i$'s generate $H$ from the vector $[0]$. Since both $\Psi$ and $s_*$ commute with the $f_i$'s, we will therefore have the identity of operators $\Psi=\alpha\cdot s_*$.
\end{proof}\\

\subsection{} Let $S^{mn}$ denote the $mn-$th symmetric power of the dual to the standard representation of $\fsl_n$ in $\BC^n$. As a vector space, $S^{mn}$ consists of homogeneous polynomials of degree
$mn$ in variables $y_1,...,y_n$. The $\fsl_n-$action on $S^{mn}$ is given by
$$
e_i=-y_{i+1} \frac {\partial}{\partial y_{i}}, \textrm{ }\textbf{ } f_i=-y_{i}
\frac {\partial}{\partial y_{i+1}}
$$
$$
h_i=y_{i+1} \frac {\partial}{\partial y_{i+1}}-y_{i} \frac {\partial}{\partial y_{i}}
$$
There is a map $\tilde{ev}:S^{mn}\rightarrow \BC$ given by evaluating polynomials at $y_1=...=y_{n-1}=0, y_n=1$. Then $\tilde{ev}$ annihilates all monomials except for $y_n^{mn}$. If we give $\BC$ the trivial $\fsl_n-$action, then one easily computes the fact that $\tilde{ev}$ commutes with the actions of $f_1,...,f_{n-1},e_1,...,e_{n-2}$. Moreover, we have that $[\tilde{ev}, h_{n-1}]=mn\cdot \tilde{ev}$. Therefore the map
\begin{equation}
\textrm{Id}\otimes \tilde{ev}:H\otimes S^{mn}\rightarrow H\otimes \BC=H
\end{equation}
commutes with the actions of $f_1,...,f_{n-1},e_1,...,e_{n-2}$ and satisfies $[\textrm{Id}\otimes \tilde{ev}, h_{n-1}]=mn\cdot \textrm{Id}\otimes \tilde{ev}$. As in ~\cite{EK}, up to a constant multiple there is a unique $\fsl_n-$intertwiner $\Phi_m:H\rightarrow H\otimes S^{mn}$. By the above, the composition
\begin{equation}
(\textrm{Id}\otimes \tilde{ev})\cdot \Phi_m:H\rightarrow H
\end{equation}
commutes with the actions of $f_1,...,f_{n-1},e_1,...,e_{n-2}$ and satisfies $[(\textrm{Id}\otimes \tilde{ev})\cdot \Phi_m, h_{n-1}]=mn\cdot (\textrm{Id}\otimes \tilde{ev})\cdot \Phi_m$. Therefore, Proposition~\ref{s_* unique} immediately implies the following:
\begin{proposition} \label{computation of s_*}
We have
$$
s_*=\alpha \cdot (\emph{Id}\otimes \tilde{ev})\cdot \Phi_m
$$
for some constant $\alpha\in \BC$. \\
\end{proposition}

\subsection{} From Proposition~\ref{factorization A(m)}, we have that $A(m)=y\cdot g\cdot s_*$. Proposition~\ref{computation of s_*} implies that
\begin{equation} \label{qwerty}
A(m)=(y\alpha)\cdot g\cdot (\textrm{Id}\otimes \tilde{ev})\cdot \Phi_m
\end{equation}
Since $g\in SL_n$, it acts in a group-like fashion on tensor products, and thus
\begin{equation}
g\cdot (\textrm{Id}\otimes \tilde{ev})\cdot g^{-1}=\textrm{Id}\otimes (g\cdot \tilde{ev}\cdot g^{-1})
\end{equation}
A trivial check reveals that $g\cdot \tilde{ev}\cdot g^{-1}=ev$, where $ev:S^{mn}\rightarrow \BC$ is the linear map that evaluates polynomials at $y_1=...=y_n=1$. Therefore we have that $g\cdot (\textrm{Id}\otimes \tilde{ev})=(\textrm{Id}\otimes ev)\cdot g$. Putting this and \eqref{qwerty} together implies that
\begin{equation} \label{123}
A(m)=(y\alpha)\cdot (\textrm{Id}\otimes ev)\cdot g\cdot \Phi_m
\end{equation}
Since $\Phi_m$ denotes any intertwiner, we can absorb the constant $y\alpha$ into it, and then make it commute with $g$ in \eqref{123}. Therefore, we obtain the following:
\begin{theorem} \label{final factorization A(m)}
For generic $\frac ax$, the operator $A(m)$ has the factorization
$$
A(m):H \stackrel{g}{\longrightarrow} H \stackrel{\Phi_m}{\longrightarrow} H\otimes S^{mn} \stackrel{\emph{Id}\otimes ev}{\longrightarrow} H
$$
where $g=A(0)$ and $\Phi_m$ is an $\fsl_n-$intertwiner. \\
\end{theorem}

\section{The generating function $Z(m)$ and the Calogero-Sutherland hamiltonian}
\label{traces}

\subsection{} As in Theorem~\ref{cohomology is Verma}, we have $H\cong M(\frac ax-\rho)$. Therefore, Theorem~\ref{final factorization A(m)} gives us
$$
A(m)=(\textrm{Id}\otimes ev)\cdot \Phi_m\cdot g
$$
where $\Phi_m:M(\frac ax-\rho)\rightarrow M(\frac ax-\rho)\otimes S^{mn}$ is an
interwtiner. By Theorem 2 in ~\cite{EK}, the character of $\Phi_m$ is given by
\begin{equation} \label{kirillov}
\chi_{\Phi_m}=c\cdot Y_{\frac ax,m}\cdot \delta^{-1}\cdot (y_1...y_n)^m
\end{equation}
In the above, $Y_{\frac ax,m}$ denotes the eigenfunction of the Calogero-Sutherland hamiltonian as in \eqref{eigenfunction CS}. Moreover, $\delta$ is the Weyl denominator \eqref{Weyl denominator} and $c\in H^*_{T\times \BC^*}(pt)$ is a constant. To compute the character of $A(m)$, we will use the fact that it
is determined by its values on elements $e^h$ of the maximal torus:
\begin{equation} \label{char computation 1}
\chi_{A(m)}(e^h)=\textrm{Tr}((\textrm{Id}\otimes ev)\cdot \Phi_m \cdot ge^h)=ev(\textrm{Tr}(\Phi_m \cdot ge^h))=ev(\chi_{\Phi_m}(ge^h))
\end{equation}
If we denote $e^h=\textrm{diag}(e^{h_1},...,e^{h_n})$, then
$$
ge^h=\left(
                                                       \begin{array}{ccccc}
                                                         e^{h_1} & e^{h_2} & e^{h_3} & ... & e^{h_n}  \\
                                                         0 & e^{h_2} & e^{h_3} & ... & e^{h_n}  \\
                                                         0 & 0 & e^{h_3} & ... & e^{h_n}  \\
                                                         ... & ... & ... & ... &...  \\
                                                         0 & 0 & 0 & ... & e^{h_n}  \\
                                                       \end{array}
                                                     \right)
$$
One easily diagonalizes the above matrix and notices that $ge^h=x^{-1}e^hx$, where the matrix $x$ is unipotent upper triangular, whose entry on row $i$ and column $j$ is
$$
x_{ij}=\frac 1{e^{h_i-h_j}-1}\cdot \prod_{k=i+1}^{j-1}\frac 1{1-e^{h_k-h_i}}
$$
The equivariance of $\chi_{\Phi_m}$ (relation \eqref{equiv}) and \eqref{char computation 1} yield
$$
\chi_{A(m)}(e^h)=ev(\chi_{\Phi_m}(x^{-1}e^h x))=ev(x^{-1}\cdot \chi_{\Phi_m}(e^h))
$$
By \eqref{kirillov}, the above becomes
\begin{equation} \label{char computation 3}
\chi_{A(m)}(e^h)=c\cdot Y_{\frac ax,m}(e^h)\cdot \delta^{-1}(e^h)\cdot ev(x^{-1}\cdot (y_1...y_n)^m)
\end{equation}
Since the representation $S^{mn}$ is nothing other than the $mn-$th symmetric power of the dual of $\BC^n$, to compute $ev(x^{-1}\cdot (y_1...y_n)^m)$ we must apply the matrix $x$ to the vector $(1,1,...,1)$, multiply the entries and raise the result to the $m-$th power. Thus we find that
$$
ev(x^{-1}\cdot (y_1...y_n)^m)=\left(\prod_{n\geq j>i\geq 1} \frac 1{1-e^{h_j-h_i}}\right)^m =
$$
$$
=\left(\prod_{\alpha \in R^+} \frac 1{1-e^{\langle h,-\alpha \rangle}}\right)^m=\left(e^{\langle h,\rho\rangle}\prod_{\alpha \in R^+} \frac 1{e^{\langle h,\alpha/2 \rangle}-e^{\langle h,-\alpha/2 \rangle}}\right)^m
=\frac {e^{\langle h,m\rho\rangle}}{\delta^m(e^h)}
$$
Plugging thus into \eqref{char computation 3} gives us the equality of formal power series
\begin{equation} \label{character of A(m)}
\chi_{A(m)}=c\cdot Y_{\frac ax,m}\cdot \delta^{-1}\cdot\frac {e^{m\rho}}{\delta^m}=c\cdot e^{m\rho}\cdot Y_{\frac ax,m}\cdot \delta^{-m-1}
\end{equation}
From the above relation, the constant $c$ is precisely the coefficient of $e^{\frac ax-\rho}$ in the power series $\chi_{A(m)}$. By \eqref{chi A(m)}, this equals
$$
c=\prod_{w\in \CT_{0} \CM}  \frac {w+mx}{w}
$$
where $0$ is the torus fixed flag with all degrees 0. But the point $0$ is isolated in $\CM$ (since the component in which it lies has dimension 0), and therefore $c=1$.  Thus \eqref{character of A(m)} and Proposition~\ref{character and gen function} imply the following:

\begin{theorem} \label{final computation character}
The generating function $Z(m)$ satisfies
\begin{equation} \label{final formula Z(m)}
Z(m)=e^{-\frac ax+(m+1)\rho}\cdot Y_{\frac ax,m}\cdot \delta^{-m-1}=Y_{\frac ax,m}\cdot e^{-\frac ax}\cdot \left(\prod_{\alpha\in R^+} \frac 1{1-e^{-\alpha}}\right)^{m+1}
\end{equation}
Here $Y_{\frac ax,m}$ is the eigenfunction of the Calogero-Sutherland hamiltonian
$$
L(m)\cdot Y_{\frac ax,m}=\frac {(a,a)}{x^2}\cdot Y_{\frac ax,m}
$$
with highest term $e^{\frac ax}$.
\end{theorem}

\begin{remark} Technically, the above proof of Theorem~\ref{final computation character} only holds for natural numbers $m$ and for generic $\frac ax$. However, the coefficients of $Z(m)$ and $Y_{\frac ax,m}$ are rational functions of both $m$ and $\frac ax$. Therefore, since the equality \eqref{final formula Z(m)} holds for infinitely many $m$ and for generic $\frac ax$, it holds identically. Similarly, once the factorization Theorem~\ref{final factorization A(m)} is proven for natural numbers $m$, it holds identically in $m$ (one must just be a bit careful with the definition of $S^{mn}$ for non-natural $m$, which will be an infinite dimensional space). \\
\end{remark}

\subsection{}
\label{connection to Braverman}

In the rest of the paper, we will take the limit $m\rightarrow \infty$ of Theorem~\ref{final computation character} and obtain the finite-dimensional statement of Corollary 3.7 of ~\cite{B1}. Let us henceforth write $m=e^Px^{-1}$, where $P\in \BC$. To simplify notations, we will write $Y_P=Y_{\frac ax,m}$. Consider the element $\rho^{\vee}\in \fh$ such that $\langle\rho^{\vee},\alpha_i\rangle=1$ for any simple root $\alpha_i$. In this section we will use two sets of coordinates on $\fh$: the \emph{old coordinate} $h$, and the \emph{new coordinate} given by $h'=h+2P \rho^{\vee}$. Under this coordinate transformation, the symbol $e^{\lambda}$ transforms to $e^{\lambda}\cdot e^{2P\langle \rho^{\vee}, \lambda\rangle}$.
Recall from Section~\ref{differential operators} that in the old coordinate $h$ we have
$$
Y_P=\sum_{\lambda \in \frac ax+Q^-} e^{\lambda} c_{\lambda,P}
$$
where we write $c_{\lambda,P}$ to emphasize the dependence on $P$. As in Section~\ref{differential operators}, the coefficients $c_{\lambda,P}$ are given by $c_{\frac ax,P}=1$ and
\begin{equation} \label{recurrence}
c_{\lambda,P}\cdot \left((\lambda,\lambda)-\frac {(a,a)}{x^2}\right)=\frac {2e^P}x\left(\frac {e^P}x+1\right)\sum_{\alpha\in
R^+}\sum_{j\geq 1} j \cdot c_{\lambda+j\alpha,P}
\end{equation}
We assume generic $\dsp \frac ax$. For $\lambda=\dsp \frac ax-(d_1\alpha_1+...+d_{n-1}\alpha_{n-1})$ we have
$$
d_1+...+d_{n-1}=\langle\rho^{\vee}, -\lambda\rangle+\langle \rho^{\vee}, \frac ax\rangle.
$$
From \eqref{recurrence}, one can prove by induction on $\lambda$ that $c_{\lambda,P}\sim e^{2P(d_1+...+d_{n-1})}$, i.e.
\begin{equation} \label{convergence}
\lim_{P\rightarrow \infty} \dsp \frac {c_{\lambda,P}}{e^{2\langle\rho^{\vee}, -\lambda\rangle}\cdot e^{2\langle \rho^{\vee}, \frac ax\rangle}}=\lim_{P\rightarrow \infty} \dsp \frac {c_{\lambda,P}}{e^{2P(d_1+...+d_{n-1})}}=:c_{\lambda} \in \BC
\end{equation}
Writing out the power series $Y_P$ in the new coordinate $h'$ gives us
$$
Y_P=e^{2P\langle \rho^{\vee}, \frac ax\rangle}\cdot \sum_{\frac ax+\lambda \in Q^-} e^{\lambda} \tilde{c}_{\lambda,P}
$$
where $\tilde{c}_{\lambda,P}=c_{\lambda,P}\cdot e^{2P\langle \rho^{\vee}, \lambda\rangle}\cdot e^{-2P\langle \rho^{\vee}, \frac ax\rangle}$. Therefore relation \eqref{convergence} implies that
$$
\lim_{P\rightarrow \infty} \tilde{c}_{\lambda,P}=c_{\lambda}
$$
This implies that we have the following termwise convergence of power series, in the new coordinate:
\begin{equation} \label{limit}
\lim_{P\rightarrow \infty} e^{-2P\langle \rho^{\vee}, \frac ax\rangle}\cdot Y_P=\sum_{\lambda \in \frac ax+Q^-} e^{\lambda}c_{\lambda}=:Y
\end{equation}
Since in the new coordinate $h'$ the functions $e^{-2P\langle \rho^{\vee}, \frac ax\rangle}\cdot Y_P$ have highest term $e^{\frac ax}$, then $Y$ will be a power series with highest term $e^{\frac ax}$. As described in ~\cite{E1}, in the new coordinate $h'$ we have
$$
L(e^P)=\Delta_{\fh}- \frac {2e^P}x\left(\dsp \frac {e^P}x+1\right)\sum_{\alpha \in R^+} \dsp \frac {1}{e^{\alpha}\cdot e^{2P\langle \rho^{\vee}, \alpha\rangle}(1-e^{-\alpha}\cdot e^{-2P\langle \rho^{\vee}, \alpha\rangle})^2}
$$
For $\alpha\in R^+$ we have $\langle \rho^{\vee}, \alpha\rangle \geq 1$, with equality if and only if $\alpha$ is a simple root. Thus as $P\rightarrow \infty$, the summands in the above expression only survive if $\alpha$ is a simple root, otherwise they go to zero. Therefore, in the new coordinate,
$$
L=\textrm{lim}_{P\rightarrow \infty} L(e^P)=\Delta_{\fh}-\dsp \frac 2{x^2}\sum_{\alpha \textrm{ simple}} e^{-\alpha}
$$
This operator $L$ is precisely the \emph{quantum Toda Hamiltonian} (see ~\cite{B1}). Because the eigenfunctions $e^{-2P\langle \rho^{\vee}, \frac ax\rangle}\cdot Y_P$ converge as $P\rightarrow \infty$, their limit $Y$ will be the unique eigenfunction of $L$ with eigenvalue $\frac {(a,a)}{x^2}$ and highest term $e^{\frac ax}$. \\

\subsection{} 

Keep the notation $m=e^Px^{-1}$ from the previous section. Recall from \eqref{coeffs Z(m)} that (in the old coordinate $h$):
$$
Z(e^Px^{-1})=\sum_{\gamma\in Q^-} e^{\gamma} \sum_{d\in \CM_{\gamma}^{T\times \BC^*}} \prod_{w\in \CT_{d} \CM}  \frac {w+e^P}{w}
$$
In the new coordinate $h'$, the above becomes:
$$
Z(e^Px^{-1})=\sum_{\gamma\in Q^-} e^{\gamma} \cdot e^{2P\langle \rho^{\vee}, \gamma\rangle}\sum_{d\in \CM_{\gamma}^{T\times \BC^*}} \prod_{w\in \CT_{d} \CM}  \frac {w+e^P}{w}
$$
But $2\langle \rho^{\vee}, -\gamma\rangle$ is precisely the dimension of $\CM_{\gamma}$, and therefore
$$
Z(e^Px^{-1})=\sum_{\gamma\in Q^-} e^{\gamma} \sum_{d\in \CM_{\gamma}^{T\times \BC^*}} \prod_{w\in \CT_{d} \CM} \frac {w+e^P}{we^P}
$$
Letting $P\rightarrow \infty$ we notice that the above series converges termwise
$$
\lim_{P\rightarrow \infty} Z(e^Px^{-1})=\sum_{\gamma\in Q^-} e^{\gamma} \sum_{d\in \CM_{\gamma}^{T\times \BC^*}} \prod_{w\in \CT_{d} \CM}  \frac {1}{w}=\sum_{\gamma\in Q^-} e^{\gamma} \int_{\CM_{\gamma}} 1 =Z
$$
where $Z$ is given by \eqref{function Z}. In the above relation, the middle equality is proven by the equivariant integration result in Corollary~\ref{integral formula}. In the old coordinate $h$, Theorem~\ref{final computation character} tells us that
$$
Z(e^Px^{-1})=Y_P\cdot e^{-\frac ax}\cdot \dsp \left(\prod_{\alpha\in R^+} \frac 1{1-e^{-\alpha}}\right)^{e^Px^{-1}+1}
$$
In the new coordinate, this equality becomes
$$
Z(e^Px^{-1})=Y_P\cdot e^{-\frac ax}\cdot e^{-2P\langle \rho^{\vee}, \frac ax\rangle}\cdot \dsp \left(\prod_{\alpha\in R^+} \frac 1{1-e^{-\alpha}\cdot e^{-2P\langle \rho^{\vee}, \alpha\rangle}}\right)^{e^Px^{-1}+1}
$$
Since $\langle \rho^{\vee}, \alpha\rangle \geq 1$ for every $\alpha\in R^+$, as $P\rightarrow \infty$ the last factor in the above product converges to 1. Taking limits as $P\rightarrow \infty$ in the above relation therefore gives us
$$
Z=\lim_{P\rightarrow \infty} Z(e^Px^{-1})=e^{-\frac ax}\cdot\lim_{P\rightarrow \infty} e^{-2P\langle \rho^{\vee}, \frac ax\rangle}\cdot Y_P
$$
By equation \eqref{limit}, the right hand side in the above is nothing but $e^{-\frac ax}\cdot Y$. Therefore, we obtain the finite-dimensional statement of Corollary 3.7 in ~\cite{B1}:

\begin{corollary} \label{conn to Braverman}
The generating function $Z$ of \eqref{function Z} equals $Y\cdot e^{-\frac ax}$, where $Y$ is the eigenfunction of the quantum Toda hamiltonian with eigenvalue $ \frac {(a,a)}{x^2}$ and highest term $e^{\frac ax}$.
\end{corollary}

\end{document}